\documentclass[12pt,draft]{article}

\usepackage[]{amsmath, amssymb,amsfonts,amsthm,mathtools,braket,color,enumerate}

\usepackage[margin=3cm]{geometry}

\usepackage{here}
\usepackage{tikz-cd}
\usepackage{tikz}
\usetikzlibrary{arrows}
\tikzset{
v/.style={
  circle, draw, inner sep=2pt, minimum size=3pt, fill=black}
}

\theoremstyle{plain}
\newtheorem{theorem}{Theorem}[section]
\newtheorem{lemma}[theorem]{Lemma}
\newtheorem{proposition}[theorem]{Proposition}
\newtheorem{corollary}[theorem]{Corollary}

\theoremstyle{definition}
\newtheorem{definition}[theorem]{Definition}

\newtheorem{example}[theorem]{Example}

\newtheorem{problem}[theorem]{Problem}

\theoremstyle{remark}

\newcommand\A{\mathcal{A}}

\DeclareMathOperator{\Der}{Der}
\DeclareMathOperator{\Ker}{Ker}
\DeclareMathOperator{\rank}{rank}
\DeclareMathOperator{\Preds}{Preds}

\title {Signed graphs and the freeness of the Weyl subarrangements of type $B_{\ell}$
}
\author{
Daisuke Suyama\thanks{
Department of Mathematics, Hokkaido University, Sapporo, Hokkaido 060-0810, Japan.
email:dsuyama@math.sci.hokudai.ac.jp}, 
Michele Torielli\thanks{Department of Mathematics, Hokkaido University, Sapporo, Hokkaido 060-0810, Japan.
email:torielli@math.sci.hokudai.ac.jp},
Shuhei Tsujie\thanks{Department of Mathematics, Hokkaido University, Sapporo, Hokkaido 060-0810, Japan.
email:tsujie@math.sci.hokudai.ac.jp}
}
\date{}

\begin{document}
\maketitle
	
\begin{abstract}
A Weyl arrangement is the hyperplane arrangement defined by a root system.
Saito proved that every Weyl arrangement is free.
The Weyl subarrangements of type $A_{\ell}$ are represented by simple graphs.
Stanley gave a characterization of freeness for this type of arrangements in terms of thier graph.
In addition, The Weyl subarrangements of type $B_{\ell}$ can be represented by signed graphs.
A characterization of freeness for them is not known. 
However, characterizations of freeness for a few restricted classes are known. 
For instance, Edelman and Reiner characterized the freeness of the arrangements between type $ A_{\ell-1} $ and type $ B_{\ell} $. 
In this paper, we give a characterization of the freeness and supersolvability of the Weyl subarrangements of type $B_{\ell}$ under certain assumption.
\end{abstract}



\section{Introduction}\label{sec:introduction}
A (central) \textbf{hyperplane arrangement} (or simply an arrangement) is a finite collection of linear hyperplanes, that is, subspaces of codimension one in a vector space. 
In this paper, we will focus on \textbf{freeness} and \textbf{supersolvability} of arrangements. 
(See Section \ref{sec:arrangement} for the definitions and results of free and supersolvable arrangements.) 

A Weyl arrangement is a collection of the hyperplanes orthogonal to the roots of a root system. 
Saito \cite{saito1977unifomization, saito1980theory} proved that Weyl arrangements are free. 

For a root system of type $ A_{\ell-1} $, the corresponding Weyl arrangement is also known as the braid arrangement. 
We can associate to each simple graph $ G $ on vertex set $ \{1, \dots, \ell \} $ a subarrangement of the braid arrangement 
\begin{align*}
\mathcal{A}(G)\coloneqq \Set{\{x_{i}-x_{j}=0\} | \{i,j\} \text{ is an edge of } G}, 
\end{align*}
where $ x_{1}, \dots, x_{\ell} $ denote the coordinates and $ \{\alpha = 0\} $ denotes the hyperplane $ \Ker(\alpha) $ for a linear form $ \alpha $. 
The arrangement $ \mathcal{A}(G) $ is called a \textbf{graphic arrangement}. 
Obviously, every subarrangement of a braid arrangement is a graphic arrangement. 
A simple graph $ G $ is called \textbf{chordal} if every cycle of length at least four has a \textbf{chord}, which is an edge of $ G $ that is not part of the cycle but connects two vertices of the cycle. 
Stanley characterized freeness and supersolvability of graphic arrangements as follows: 
\begin{theorem}[Stanley {\cite[Corollary 4.10]{stanley2007introduction}} ]\label{Stanley graphic arrangement}
For a simple graph $ G $, the following are equivalent
\begin{enumerate}[(1)]
\item $ G $ is chordal. 
\item $ \mathcal{A}(G) $ is supersolvable. 
\item $ \mathcal{A}(G) $ is free. 
\end{enumerate}
\end{theorem}
Thus, the freeness of Weyl subarrangements of type $ A_{\ell} $ is completely characterized in combinatorial terms. 
For a Weyl arrangement of arbitrary type, only sufficient condition for freeness is known by \cite{abe2016freeness-jotems}. 
Except Weyl arrangements of type $A_{\ell}$, any characterization of the freeness of Weyl subarrangements is not known. 

In this paper, we study the freeness of Weyl subarrangements of type $ B_{\ell} $. 
For our purpose, we define a signed graph and the corresponding arrangement as follows. 
See Zaslavsky \cite{zaslavsky1982signed-dam} for a general treatment of signed graphs.
\begin{definition}
A \textbf{signed graph} is a tuple $ G=(G^{+},G^{-},L_{G}) $, where $ G^{+}=(V_{G},E_{G}^{+}) $ and $ G^{-}=(V_{G},E_{G}^{-}) $ are simple graphs on the same vertex set $ V_{G} $ and $ L_{G} $ is a subset of $ V_{G} $, which is called the set of \textbf{loops}. 
An element of $ E_{G}^{+} $ (resp. $ E_{G}^{-} $) is called a \textbf{positive edge} (resp. a \textbf{negative edge}). 
When we do not consider loops, we write a signed graph as $ (G^{+}, G^{-}) $. 
\end{definition}
\begin{definition}
Let $ G=(G^{+},G^{-},L_{G}) $ be a signed graph on vertex set $ \{1, \dots, \ell\} $. 
Define the \textbf{signed-graphic arrangement} $ \mathcal{A}(G) $ in the $ \ell $-dimensional vector space over $ \mathbb{R} $ (or any field of characteristic zero) by 
\begin{multline*}
\mathcal{A}(G) \coloneqq \Set{\{x_{i}-x_{j}=0 | \{i,j\} \in E_{G}^{+} \}} \\
\cup \Set{\{x_{i}+x_{j}=0 | \{i,j\} \in E_{G}^{-} \}} 
\cup \Set{ \{x_{i}=0\} | i \in L_{G}}. 
\end{multline*}
\end{definition}
Note that every simple graph $ G $ on $ \ell $ vertices can be naturally regarded as a signed graph $ G=(G, \overline{K}_{\ell}, \varnothing) $, where $ \overline{K}_{\ell} $ denotes the edgeless graph on $ \ell $ vertices, and hence a signed-graphic arrangement is a generalization of a graphic arrangement. 

We call the signed graph $ B_{\ell}\coloneqq(K_{\ell},K_{\ell},\{1,\dots,\ell\}) $ the \textbf{complete signed graph with loops}, where $ K_{\ell} $ denotes the complete graph on $ \ell $ vertices.  
The corresponding signed-graphic arrangement $ \mathcal{A}(B_{\ell}) $ is the Weyl arrangement of type $ B_{\ell} $. 
Moreover, every subarrangement of $ \mathcal{A}(B_{\ell}) $ is a signed-graphic arrangement of some signed graph, and vise versa. 

Zaslavsky (Theorem \ref{Zaslavsky SS}) characterized supersolvability of signed-graphic arrangements. 
Edelman and Reiner characterized freeness and supersolvability of arrangements between type $ A_{\ell-1} $ and type $ B_{\ell} $. 
The signed graphs $ G $ corresponding to these arrangements satisfy $ G^{+} = K_{\ell} $. 
Figure \ref{Fig:known results} summarizes the known results for supersolvability and freeness. 

\begin{figure}[t]
\centering
\begin{tabular}{|l|l|}
\hline
Condition & Reference \\ 
\hline
\rule{0pt}{15pt}$ G^{-}=\overline{K}_{\ell}, L_{G}=\varnothing $ & Stanley, Theorem \ref{Stanley graphic arrangement} \\
$ G^{-}=\overline{K}_{\ell} $ & Corollary to \cite[Theorem 1.4]{suyama2015vertex-weighted-a} \\
$ G^{+}=K_{\ell} $ & Edelman and Reiner, Theorem \ref{ER free} and \cite[Theorem 4.15]{edelman1994free} \\
$ G^{+}=\overline{K}_{\ell} $ & Bailey \cite[Theorem 4.2]{bailey????inductively-p} \\
$ G^{-}=K_{\ell} $ & Bailey \cite[Theorem 4.3]{bailey????inductively-p} \\
$ G^{+} = G^{-} $ & Bailey \cite[Theorem 4.4]{bailey????inductively-p} \\
None & Zaslavsky, Theorem \ref{Zaslavsky SS} only for supersolvability \\
\hline
\end{tabular}
\caption{Summary of known results for freeness and supersolvablility}
\label{Fig:known results}
\end{figure}

A \textbf{cycle} of length $ k \geq 3 $ (shortly $ k $-cycle) is a sequence of distinct vertices $ v_{1}, \dots , v_{k} $ with edges $ \{v_{1},v_{2}\}, \dots , \{v_{k-1}, v_{k}\}, \{v_{k},v_{1}\} $ (positive or negative edges are allowed). 
A cycle is called \textbf{balanced} if it has an even number of negative edges. 
Otherwise, we call it \textbf{unbalanced}. 
A signed graph $ G $ is called \textbf{balanced chordal} if every balanced cycle of length at least four has a \textbf{balanced chord}, which is an edge of $ G $ that is not part of the cycle but connects two vertices of the cycle, and moreover that separates the cycle into two balanced cycles. 
Note that balanced chordality has nothing to do with the loop set. 

Our main result is as follows: 
\begin{theorem}\label{main theorem}
Let $ G=(G^{+},G^{-}) $ be a signed graph on vertex set $ V_{G} $. 
Assume that $ G^{+} \supseteq G^{-} $. 
Then the following are equivalent: 
\begin{enumerate}[(1)]
\item \label{main theorem 1} $ G $ is balanced chordal. 
\item \label{main theorem 2} $ \mathcal{A}(G^{+}, G^{-}, V_{G}) $ is free. 
\item \label{main theorem 3} $ \mathcal{A}(G^{+}, G^{-}, L) $ is free for some loop set $ L \subseteq V_{G} $. 
\end{enumerate}
\end{theorem}

The organization of this paper is as follows. 
In Section \ref{sec:arrangement}, we give basic definitions and results for free and supersolvable arrangements. 
In Section \ref{sec:simple graphs}, 
we introduce theorems for simple graphs which are required in this paper. 
In Section \ref{sec:signed graphs}, we study singed graphs and the corresponding signed-graphic arrangements. 
In Section \ref{sec:divisional vertex}, we introduce the notion of divisional edges and vertices of signed graphs. 
In Section \ref{sec:proof}, we introduce several lemmas and give a proof of Theorem \ref{main theorem}.

\section{Review of free and supersolvable arrangements}\label{sec:arrangement}
In this section, we review some basic concepts on arrangements. 
A standard reference for the theory of arrangements is \cite{orlik1992arrangements}. 
Throughout this section, the ambient space of an arrangement is the $ \ell $-dimensional vector space $ \mathbb{K}^{\ell} $ over an arbitrary field $ \mathbb{K} $. 
Let $S$ be the symmetric algebra of $ (\mathbb{K}^{\ell})^{\ast} $, which can be identified with the polynomial ring $\mathbb{K}[x_{1},\ldots ,x_{\ell}]$, where $\{x_{1},\ldots ,x_{\ell}\}$ is a basis for $(\mathbb{K}^{\ell})^{*}$.
Let $\Der (S)$ denote the module of derivations of $S$:
\[
\Der(S) \coloneqq \{ \theta \colon S \rightarrow S \mid \theta \text{ is } \mathbb{K} \text{-linear}, \\
\theta (fg)=\theta (f)g+f\theta (g) \text{ for } f,g \in S \}.
\]
For an arrangement $\mathcal{A}$ in $ \mathbb{K}^{\ell} $, the module of logarithmic derivations $D(\mathcal{A})$ of $\mathcal{A}$ is defined by
\begin{align*}
D(\A) &\coloneqq \{ \theta \in \Der(S) \mid \theta (Q(\mathcal{A})) \in Q(\mathcal{A})S \} \\
&= \{ \theta \in \Der(S) \mid \theta(\alpha_{H}) \in \alpha_{H}S \text{ for } H \in \mathcal{A} \},
\end{align*} 
where $\alpha_{H}$ is a linear form such that $\Ker (\alpha_{H})=H$ and $Q(\mathcal{A}) \coloneqq \prod_{H \in \mathcal{A}}\alpha_{H}$ is the defining polynomial of $\mathcal{A}$. 
\begin{definition}
An arrangement $ \mathcal{A} $ is called free if $ D(\mathcal{A}) $ is a free $ S $-module. 
\end{definition}
When $\mathcal{A}$ is free, the module $D(\mathcal{A})$ has a homogeneous basis
$\{ \theta_{1},\ldots ,\theta_{\ell} \}$ and
the degrees $ \deg \theta_{1},\ldots ,\deg \theta_{\ell} $ are called
the \textbf{exponents} of $\mathcal{A}$. 

The rank of an arrangement $ \mathcal{A} $, denoted by $ \rank(\mathcal{A}) $, is the codimension of $ \bigcap_{H \in \mathcal{A}}H $. 
The \textbf{intersection lattice} $L(\mathcal{A})$ is the set of all intersections of hyperplanes in $\mathcal{A}$, which is partially ordered by reverse inclusion: $ X \leq Y \Leftrightarrow Y \subseteq X $. 
We say that $\mathcal{A}$ is \textbf{supersolvable} if $L(\mathcal{A})$ is supersolvable as defined by Stanley \cite{stanley1972supersolvable}.
In this paper, we omit the definition in detail. 
However, supersolvability of arrangements is characterized as follows: 
\begin{theorem}[{Bj{\"o}rner-Edelman-Ziegler \cite[Theorem 4.3]{bjorner1990hyperplane-dcg}}]\label{BEZ}
An arrangement $ \A $ is supersolvable if and only if
there exists a filtration
\begin{align*}
\A = \A_{r} \supseteq \A_{r-1} \supseteq \dots \supseteq \A_{1}
\end{align*}
such that
\begin{enumerate}[(1)]
\item $ \rank(\A_{i}) = i \quad (i=1,2,\dots, r) $.
\item For every $ i \geq 2 $ and any distinct hyperplanes $ H, H^{\prime} \in \A_{i} \setminus \mathcal{A}_{i-1} $, there exists some $ H^{\prime\prime} \in \A_{i-1} $ such that $ H \cap H^{\prime} \subseteq H^{\prime\prime} $. 
\end{enumerate}
\end{theorem}

Jambu and Terao revealed a relation between supersolvability and freeness. 
\begin{theorem}[Jambu-Terao {\cite[Theorem 4.2]{jambu1984free-aim}}]\label{JT SS=>free}
Every supersolvable arrangement is free. 
\end{theorem}
Supersolvability is a combinatorial property, that is, it is determined by the intersection lattice. 
It is conjectured that freeness of arrangements for a fixed field is also a combinatorial property (Terao Conjecture \cite{terao1983exponents}). 

An arrangement $ \mathcal{A} $ is said to be \textbf{independent} if $ \rank(\mathcal{A}) = |\mathcal{A}| $. 
Call $ \mathcal{A} $ \textbf{dependent} if $ \rank(\mathcal{A}) < |\mathcal{A}| $. 
It is easy to show that every independent arrangement is supersolvable and hence free. 
An arrangement $ \mathcal{A} $ is called \textbf{generic} if every subarrangement of cardinality $ \rank(\mathcal{A}) $ is independent. 
The following theorem states that, except for trivial cases, generic arrangements are non-free. 
\begin{theorem}[Rose-Terao {\cite{rose1991free-joa}}, Yuzvinsky {\cite{yuzvinsky1991free-joa}}]\label{RT_Y generic}
Let $ \mathcal{A} $ be a generic arrangement. 
Suppose that $ |\mathcal{A}| > \rank(\mathcal{A}) \geq 3 $. 
Then $ \mathcal{A} $ is non-free. 
\end{theorem}

An arrangement $ \mathcal{A} $ is called a \textbf{circuit} if $ \mathcal{A} $ is minimally dependent, that is, $ \mathcal{A} $ is dependent but $ \mathcal{A} \setminus \{H\} $ is independent for any $ H \in \mathcal{A} $. 
This terminology stems from matroid theory. 
We obtain the following corollary of Theorem \ref{RT_Y generic}. 
\begin{corollary}\label{circuit nonfree}
If an arrangement $ \mathcal{A} $ is a circuit, then $ \mathcal{A} $ is generic. 
Moreover, if $ |\mathcal{A}| \geq 4 $, then $ \mathcal{A} $ is non-free. 
\end{corollary}

For every arrangement $ \mathcal{A} $, the one-variable \textbf{M{\"o}bius function} $ \mu(X) $ on $ L(\mathcal{A}) $ is defined recursively by 
\begin{align*}
\sum_{Y \leq X} \mu(Y) = \delta_{\hat{0} \, X}, 
\end{align*}
where $ \delta_{\hat{0} \, X} $ denotes the Kronecker delta and $ \hat{0} $ denotes the minimal element of $ L(\mathcal{A}) $, namely the ambient space. 
Moreover, we can associate to $ \mathcal{A} $ a polynomial $ \chi(\mathcal{A}, t) \in \mathbb{Z}[t] $, called the \textbf{characteristic polynomial}, defined by 
\begin{align*}
\chi(\mathcal{A},t) \coloneqq \sum_{X \in L(\mathcal{A})} \mu(X)t^{\dim X}.  
\end{align*}
The characteristic polynomial of an arrangement is one of the most important invariants. 
Terao \cite{terao1981generalized-im} showed that the characteristic polynomial of a free arrangement can be factored into a product of linear factors over $ \mathbb{Z} $ with non-negative roots. 

For an element $ X \in L(\mathcal{A}) $, we define the \textbf{localization} $ \mathcal{A}_{X} $ and the \textbf{restriction} $ \mathcal{A}^{X} $ by 
\begin{align*}
\mathcal{A}_{X} &\coloneqq \Set{H \in \mathcal{A} | H \supseteq X }, \\
\mathcal{A}^{X} &\coloneqq \Set{H \cap X | H \in \mathcal{A} \setminus \mathcal{A}_{X}}. 
\end{align*}

\begin{proposition}[Stanley {\cite[Proposition 3.2]{stanley1972supersolvable}}]\label{Stanley SS lattice}
Every localization and every restriction of a supersolvable arrangement is supersolvable. 
\end{proposition}

The following results are quite useful for determining whether an arrangement is free or not. 
\begin{proposition}[Orlik-Terao {\cite[Theorem 4.37]{orlik1992arrangements}}]\label{OT_localization}
Every localization of a free arrangement is free. 
\end{proposition}

\begin{theorem}[Division Theorem, Abe {\cite[Theorem 1.1]{abe2016divisionally-im}}]\label{Abe Division Theorem}
Let $ \mathcal{A} $ be an arrangement. 
Assume that there exists a hyperplane $ H \in \mathcal{A} $ such that $ \chi(\mathcal{A}^{H}, t) $ divides $ \chi(\mathcal{A},t) $ and $ \mathcal{A}^{H} $ is free. 
Then $ \mathcal{A} $ is free. 
\end{theorem}

\section{Preliminaries from simple graphs}\label{sec:simple graphs}
In this section, we will introduce some basic notions about simple graphs. 

\subsection{Threshold graphs}\label{subsec:threshold graph}
As mentioned in Section \ref{sec:introduction}, Edelman and Reiner characterized freeness of subarrangements between type $ A_{\ell-1} $ and $ B_{\ell} $. 
The characterization requires the notion of threshold graphs. 

\begin{definition}
Threshold graphs are defined recursively by the following construction: 
\begin{enumerate}[(1)]
\item The single-vertex graph $ K_{1} $ is threshold. 
\item The graph obtained by adding an isolated vertex to a threshold graph is threshold. 
\item The graph obtained by adding a dominating vertex to a threshold graph is threshold, where a \textbf{dominating vertex} is a vertex which is adjacent to all other vertices. 
\end{enumerate}
\end{definition}

The \textbf{degree} of a vertex $ v $ in a simple graph $ G $ (denoted $ \deg_{G}(v) $) is the number of incident edges. 
\begin{definition}
An ordering $ (v_{1}, \dots, v_{\ell}) $ of the vertices of a simple graph $ G $ is called a \textbf{degree ordering} if $ \deg_{G}(v_{1}) \geq \dots \geq \deg_{G}(v_{\ell}) $. 
An \textbf{initial segment} of an ordering $ (v_{1}, \dots, v_{\ell}) $ is a set $ \{v_{1}, \dots, v_{k}\} $ for some $ k $.  
\end{definition}
A simple graph may admit several degree orderings. 
However, for a threshold graph, there is only one degree ordering up to an automorphism \cite{hammer1981threshold-sjoadm}. 
We will need the following characterization for threshold graphs. 
\begin{theorem}[Golumbic {\cite[Corollary 5]{golumbic1978trivially-dm}}]\label{Golumbic}
Threshold graphs have a forbidden induced subgraph characterization. 
Namely, a simple graph is threshold if and only if it contains no induced subgraph isomorphic $ 2K_{2}, C_{4} $, or $ P_{4} $ (See Figure \ref{Fig:threshold}).
\end{theorem}

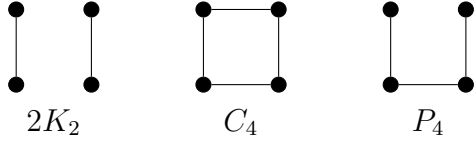
\begin{figure}[t]
\centering
\begin{tikzpicture}
\draw (0,1) node[v](1){};
\draw (0,0) node[v](2){};
\draw (1,0) node[v](3){};
\draw (1,1) node[v](4){};
\draw (0.5,-0.5) node[]{$ 2K_{2} $};
\draw (1)--(2);
\draw (3)--(4);
\end{tikzpicture}
\hspace{10mm}
\begin{tikzpicture}
\draw (0,1) node[v](1){};
\draw (0,0) node[v](2){};
\draw (1,0) node[v](3){};
\draw (1,1) node[v](4){};
\draw (0.5,-0.5) node[]{$ C_{4} $};
\draw (1)--(2)--(3)--(4)--(1);
\end{tikzpicture}
\hspace{10mm}
\begin{tikzpicture}
\draw (0,1) node[v](1){};
\draw (0,0) node[v](2){};
\draw (1,0) node[v](3){};
\draw (1,1) node[v](4){};
\draw (0.5,-0.5) node[]{$ P_{4} $};
\draw (1)--(2)--(3)--(4);
\end{tikzpicture}
\caption{Forbidden induced subgraphs for threshold graphs}\label{Fig:threshold}
\end{figure} 

\subsection{Menger's theorem}
Given a simple graph $ G=(V_{G},E_{G}) $ and a subset $ W \subseteq V_{G} $, let $ G[W] $ denote the subgraph induced by $ W $ and let $ G \setminus W $ denote $ G[V_{G} \setminus W] $. 
For non-adjacent vertices $ a,b $ which belong to the same connected component of $ G $, a non-empty subset $ S \subseteq V_{G} $ is called an \textbf{$ (a,b) $-separator} if the vertices $ a,b $ belong to distinct connected components of $ G \setminus S $. 
An $ (a,b) $-separator is \textbf{minimal} if no proper subset is $ (a,b) $-separator. 
Moreover, a subset $ S \subseteq V_{G} $ is called a \textbf{minimal vertex separator} if $ S $ is a minimal $ (a,b) $-separator for some vertices $ a,b $. 
Note that a proper subset of a minimal vertex separator can be a minimal vertex separator for other vertices. 

Let $ p_{1}, \dots, p_{n} $ be paths from a vertex $ a $ to another vertex $ b $. 
The paths $ p_{1}, \dots, p_{n} $ are \textbf{internally disjoint} if they do not have any internal vertex in common. 
A path is called \textbf{induced} if it is an induced subgraph, that is, there are no edges connecting non-consecutive vertices in the path. 
\begin{theorem}[Menger {\cite{menger1927zur-fm}}]\label{Menger}
Let $ a,b $ be non-adjacent vertices of a connected graph. 
Then the minimum of the cardinalities of minimal $ (a,b) $-separators equals to the maximum number of internally disjoint induced paths from $ a $ to $ b $. 
\end{theorem}
In this paper, the following corollary is required: 
\begin{corollary}\label{Menger cor}
Let $ a,b $ be non-adjacent vertices of a connected graph $ G $ and $ S $ a minimal $ (a,b) $-separator of minimal cardinality. 
Let $ u,v $ be distinct vertices in $ S $. 
Then there exist a cycle of $ G $ such that it contains the vertices $ a,b,u,v $, it intersects $ S $ at $ \{u,v\} $, and it consists of two induced paths from $ a $ to $ b $. 
\end{corollary}
\begin{proof}
By Theorem \ref{Menger}, there exists a set $ \Set{p_{s} | s \in S} $ of  internally disjoint induced paths from $ a $ to $ b $, where the path $ p_{s} $ intersects $ S $ at $ \{ s \} $ for every $ s \in S $. 
The cycle obtained by connecting $ p_{u} $ and $ p_{v} $ at their endvertices is a desired cycle. 
\end{proof}

\subsection{Chordal graphs and their clique-separator graph}\label{subsec:clique-separator graph}
A subset $ C \subseteq V_{G} $ is called a \textbf{clique} of $ G $ if $ G[C] $ is a complete graph. 
A vertex $ v \in V_{G} $ is called \textbf{simplicial} if the neighborhood of $ v $ is a clique. 
An ordering $ (v_{1}, \dots, v_{\ell}) $ of the vertices is called a \textbf{perfect elimination ordering} if $ v_{i} $ is simplicial in $ G[\{v_{1}, \dots, v_{i}\}] $ for every $ i \in \{1,\dots, \ell\} $. 
Recall that a chordal graph is a graph whose cycles of length at least four have chords. 
The following theorems are basic results for chordal graphs. 
\begin{theorem}[Dirac {\cite[Theorem 1]{dirac1961rigid-aadmsduh}}]\label{Dirac minmal vertex separator}
A simple graph is chordal if and only if every minimal vertex separator is a clique. 
\end{theorem}
\begin{theorem}[Dirac {\cite[Theorem 4]{dirac1961rigid-aadmsduh}}]\label{Dirac simplicial}
Every chordal graph is complete or has at least two non-adjacent simplicial vertices. 
\end{theorem}
\begin{theorem}[Fulkerson-Gross {\cite[Section 7]{fulkerson1965incidence-pjom}}]\label{FG}
A simple graph is chordal if and only if it admits a perfect elimination ordering. 
\end{theorem}

Ibarra \cite{ibarra2009clique-separator-dam} has invented the clique-separator graph of a chordal graph to describe its structure. 
\begin{definition}
Let $ G $ be a chordal graph. 
The clique-separator graph $ \mathcal{G} $ of $ G $ consists of the following nodes, (directed) arcs, and (undirected) edges. 
\begin{itemize}
\item $ \mathcal{G} $ has a clique node $ C $ for each maximal clique $ C $ of $ G $. 
\item $ \mathcal{G} $ has a separator node $ S $ for each minimal vertex separator $ S $ of $ G $. 
\item An arc of $ \mathcal{G} $ is from a separator node to another separator node.  
The tuple $ (S,S^{\prime}) $ of separator nodes is an arc of $ \mathcal{G} $ if $ S \subsetneq S^{\prime} $ and there exists no separator node $ S^{\prime\prime} $ such that $ S \subsetneq S^{\prime\prime} \subsetneq S^{\prime} $. 
\item An edge of $ \mathcal{G} $ is between a clique node and a separator node. 
For a clique node $ C $ and a separator node $ S $, the set $ \{C, S\} $ is an edge of $ \mathcal{G} $ if $ S \subsetneq C $ and there exists no separator node $ S^{\prime} $ such that $ S \subsetneq S^{\prime} \subsetneq C $. 
\end{itemize}
\end{definition}
In the rest of this subsection, $ G, \mathcal{G} $ denote a chordal graph and its clique-separator graph, respectively. 
Note that we use terminologies ``vertex" for $ G $ and ``node" for $ \mathcal{G} $ after Ibarra. 
Figure \ref{Fig:clique-separator graph} shows an example of a chordal graph and its clique-separator graph, where $ C_{X}, S_{X} $ denote the maximal clique and the minimal vertex separator on a set $ X $ of labels of vertices, respectively. 
The clique-separator graph has many remarkable properties which describe the structure of a chordal graph. 
We will now introduce some of them required in this paper. 

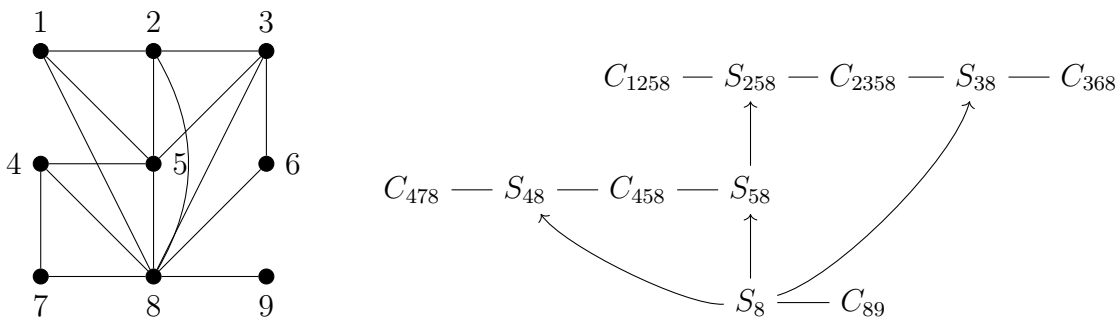
\begin{figure}[t]
\centering
\begin{tikzpicture}
\draw (  0,  3) node[v,label=above:{$ 1 $}](1){};
\draw (1.5,  3) node[v,label=above:{$ 2 $}](2){};
\draw (  3,  3) node[v,label=above:{$ 3 $}](3){};
\draw (  0,1.5) node[v,label=left:{$ 4 $}](4){};
\draw (1.5,1.5) node[v,label=right:{$ 5 $}](5){};
\draw (  3,1.5) node[v,label=right:{$ 6 $}](6){};
\draw (  0,  0) node[v,label=below:{$ 7 $}](7){};
\draw (1.5,  0) node[v,label=below:{$ 8 $}](8){};
\draw (  3,  0) node[v,label=below:{$ 9 $}](9){};
\draw (1)--(2);
\draw (1)--(5);
\draw (1)--(8);
\draw (2)--(3);
\draw (2)--(5);
\draw[bend left] (2) to (8);
\draw (3)--(5);
\draw (3)--(6);
\draw (3)--(8);
\draw (4)--(5);
\draw (4)--(7);
\draw (4)--(8);
\draw (5)--(8);
\draw (6)--(8);
\draw (7)--(8);
\draw (8)--(9);
\end{tikzpicture}\qquad
\begin{tikzpicture}
\draw (-1.5,3) node[](C1258){$ C_{1258} $};
\draw (0,3) node[](S258){$ S_{258} $};
\draw (1.5,3) node[](C2358){$ C_{2358} $};
\draw (3,3) node[](S38){$ S_{38} $};
\draw (4.5,3) node[](C368){$ C_{368} $};
\draw (-4.5,1.5) node[](C478){$ C_{478} $};
\draw (-3,1.5) node[](S48){$ S_{48} $};
\draw (-1.5,1.5) node[](C458){$ C_{458} $};
\draw (0,1.5) node[](S58){$ S_{58} $};
\draw (0,0) node[](S8){$ S_{8} $};
\draw (1.5,0) node[](C89){$ C_{89} $};
\draw (C1258)--(S258)--(C2358)--(S38)--(C368);
\draw (C478)--(S48)--(C458)--(S58);
\draw (S8)--(C89);
\draw[->, bend left, distance=5mm] (S8) to (S48);
\draw[->] (S8)--(S58);
\draw[->, bend right,distance=8mm] (S8) to (S38);
\draw[->] (S58)--(S258);
\end{tikzpicture}
\caption{A chordal graph and its clique-separator graph}\label{Fig:clique-separator graph}
\end{figure}

\begin{proposition}[Ibarra {\cite[p.1739 Observation]{ibarra2009clique-separator-dam}}]\label{Ibarra observation}
For every separator node $ S $ of $ \mathcal{G} $, one of the following holds. 
\begin{enumerate}[(1)]
\item There exist distinct clique nodes $ C, C^{\prime} $ such that $ \{S,C\} $ and $ \{S,C^{\prime}\} $ are edges of $ \mathcal{G} $. 
\item There exist distinct separator nodes $ S^{\prime},S^{\prime\prime} $ such that $ (S,S^{\prime}), (S,S^{\prime\prime}) $ are arcs of $ \mathcal{G} $. 
\item There exist a clique node $ C $ and a separator node $ S^{\prime} $ such that $ \{S,C\} $ is an edge of $ \mathcal{G} $ and $ (S,S^{\prime}) $ is an arc of $ \mathcal{G} $. 
\end{enumerate}
\end{proposition}

A \textbf{box} of the clique-separator graph $ \mathcal{G} $ is a connected components of the undirected graph obtained by deleting the arcs of $ \mathcal{G} $. 
Let $ \mathcal{G}^{c} $ denote the directed graph obtained by contracting each box into a single node and replacing multiple arcs by a single arc. 

\begin{theorem}[Ibarra {\cite[Theorem 3(2)]{ibarra2009clique-separator-dam}}]\label{Ibarra2}
The following hold: 
\begin{enumerate}[(1)]
\item \label{Ibarra2-1} Every box is a tree with the \textbf{clique intersection property}: for any distinct nodes $ N,N^{\prime} $ of a box and any internal node $ N^{\prime\prime} $ between $ N $ and $ N^{\prime} $, we have $ N \cap N^{\prime} \subseteq N^{\prime\prime} $. 
\item \label{Ibarra2-2} The separator nodes in each box form an antichain, that is, there are no inclusions among the separator nodes. 
\item \label{Ibarra2-3} $ \mathcal{G}^{c} $ is a directed acyclic graph. 
\end{enumerate}
\end{theorem}

Every directed acyclic finite graph has a sink, namely a vertex with no outgoing arcs from it. 
By Theorem \ref{Ibarra2}(\ref{Ibarra2-3}), the directed graph $ \mathcal{G}^{c} $ has a sink, which we call a \textbf{sink box} of $ \mathcal{G} $. 

\begin{proposition}\label{sink_box_leaf}
Let $ B $ be a sink box of $ \mathcal{G} $. 
Then $ B $ is a tree whose leaves are clique nodes. 
\end{proposition}
\begin{proof}
From Theorem \ref{Ibarra2}(\ref{Ibarra2-1}), the box $ B $ is a tree. 
Let $ S $ be a separator node of $ B $. 
Since $ B $ is a sink, there are no outgoing arcs from $ S $. 
Hence, by Proposition \ref{Ibarra observation}, there exist at least two clique nodes adjacent to $ S $. 
Therefore $ S $ cannot be a leaf of $ B $. 
Thus every leaf of $ B $ is a clique node. 
\end{proof} 
For every separator node $ S $ of $ \mathcal{G} $, define 
\begin{align*}
\Preds(S) \coloneqq \Set{S^{\prime} | \text{$ S^{\prime} $ is a separator node of $ \mathcal{G} $ such that $ S^{\prime} \subseteq S $}}. 
\end{align*}
For a subgraph $ \mathcal{F} $ of $ \mathcal{G} $, let $ G[\mathcal{F}] $ denote the subgraph of $ G $ induced by 
\begin{align*}
\Set{v \in V_{G} | \text{$ v $ belongs to some node of $ \mathcal{F} $}}. 
\end{align*} 

\begin{theorem}[Ibarra {\cite[Theorem 3(1)]{ibarra2009clique-separator-dam}}]\label{Ibarra 1}
Let $ S $ be a separator node of $ \mathcal{G} $. 
Suppose that $ \mathcal{G}_{1}, \dots, \mathcal{G}_{k} $ are the connected components of $ \mathcal{G} \setminus \Preds(S) $. 
Then the subgraphs $ G[\mathcal{G}_{1}] \setminus S, \dots, G[\mathcal{G}_{k}] \setminus S $ are the connected components of $ G \setminus S $. 
\end{theorem}

\begin{corollary}\label{Ibarra 1 cor}
Let $ P $ be a path of a sink box of $ \mathcal{G} $ from a clique node to another clique node indicated as follows: 
\begin{center}
\begin{tikzpicture}
\draw (0,0) node[](C1){$ C_{1} $}; 
\draw (1,0) node[](S2){$ S_{2} $};
\draw (2,0) node[](C2){$ C_{2} $};
\draw (4.6,0) node[](Sk){$ S_{k} $};
\draw (5.6,0) node[](Ck){$ C_{k} $};
\draw (C1)--(S2)--(C2)--(2.8,0);
\draw[dotted] (2.8,0)--(3.8,0);
\draw (3.8,0)--(Sk)--(Ck); 
\end{tikzpicture}
\end{center}
Take vertices $ a \in C_{1}\setminus S_{2} $ and $ b \in C_{k} \setminus S_{k} $. 
Then the set of $ (a,b) $-minimal separators of $ G $ coincides with $ \{S_{2}, \dots, S_{k}\} $. 
\end{corollary}
\begin{proof}
By Theorem \ref{Ibarra 1}, each $ S_{i} $ is a minimal $ (a,b) $-separator, and the other separator nodes cannot separate $ a $ from $ b $. 
\end{proof}


\section{Signed graphs and signed-graphic arrangements}\label{sec:signed graphs}
In this section, we introduce various notions about signed graphs which are related to freeness of signed-graphic arrangements. 

\subsection{Necessary conditions for freeness}
A \textbf{path} is a sequence $ v_{1} \dots v_{k} $ of distinct vertices with edges $ \{v_{1},v_{2}\}, \dots, \{v_{k-1},v_{k}\} $, where the sign of each edge may be either positive or negative. 
A signed graph is called \textbf{connected} if any distinct two vertices are connected by a path. 

Recall that a cycle of length $ k \geq 3 $ is a sequence of distinct vertices $ v_{1} \dots v_{k} $ with edges $ \{v_{1},v_{2}\}, \dots, \{v_{k-1},v_{k}\}, \{v_{k},v_{1}\} $ and that a cycle is called \textbf{balanced} if it has an even number of negative edges. 
Moreover, we define cycles of length $ 1 $ and $ 2 $ as follows: 
the $ 1 $-cycle is a vertex with the loop and the $ 2 $-cycle is a pair of vertices $ \{u,v\} $ with the positive and the negative edges between them. 
Figure \ref{Fig:unbalanced 1,2-cycles} illustrates these cycles, where the dashed edge denotes the negative edge.
The $ 2 $-cycle has exactly one negative edge and hence it is unbalanced.  
Every loop corresponds to a hyperplane of the form $ \{x=0\} $, which is equal to $ \{x=-x\} $. 
Therefore it is natural that a loop is assigned to be negative. 
Thus, we define the $ 1 $-cycle to be unbalanced. 
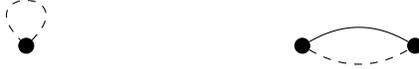
\begin{figure}[t]
\centering
\begin{tikzpicture}[baseline=0]
\draw node[v](1){};
\draw[scale=2,dashed] (1)  to[in=135,out=45,loop] (1);
\end{tikzpicture}
\hspace{25mm}
\begin{tikzpicture}[baseline=0]
\draw (0,0) node[v](1){};
\draw (1.5,0) node[v](2){};
\draw[bend left] (1) to (2);
\draw[bend right, dashed] (1) to (2);
\end{tikzpicture}
\caption{The unbalanced $ 1 $-cycle and the unbalanced $ 2 $-cycle}\label{Fig:unbalanced 1,2-cycles}
\end{figure}

Every arrangement determines a matroid on itself. 
Hence a signed-graphic arrangement $ \mathcal{A}(G) $ determines a matroid on itself and hence on the edge set $ E_{G} \coloneqq E_{G}^{+} \sqcup E_{G}^{-} \sqcup L_{G} $ of the corresponding signed graph $ G $. 
Zaslavsky has classified circuits (minimal dependent sets) of the matroid, which are called \textbf{frame circuits}, as follows (Figure \ref{Fig:exam_frame_circuits} shows examples of frame circuits): 
\begin{proposition}[Zaslavsky {\cite[8B.1]{zaslavsky1982signed-dam}},{\cite[Corollary 3.2]{zaslavsky2012signed-jociss}}]\label{Zaslavsky frame circuit}
A connected signed graph is a frame circuit if and only if it is one of the following graphs: 
\begin{enumerate}[(1)]
\item A balanced cycle. 
\item A pair of disjoint unbalanced cycles together with a path which connects them (called a \textbf{loose handcuff}). 
\item A pair of disjoint unbalanced cycles which intersect in precisely one vertex (called a \textbf{tight handcuff}). 
\end{enumerate}
\end{proposition}
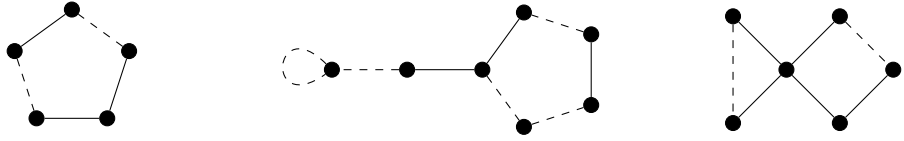
\begin{figure}[t]
\centering
\begin{tikzpicture}[baseline=0]
\foreach \x in {0,...,4}
\draw (90+72*\x:0.8) node[v](\x){};
\draw (0)--(1);
\draw[dashed] (1)--(2);
\draw (2)--(3)--(4);
\draw[dashed] (4)--(0);
\end{tikzpicture}
\hspace{14mm}
\begin{tikzpicture}[baseline=0]
\foreach \x in {0,...,4}
\draw (180+72*\x:0.8) node[v](\x){};
\draw (-1.8,0) node[v](5){};
\draw (-2.8,0) node[v](6){};
\draw[dashed] (0)--(1)--(2);
\draw (2)--(3);
\draw[dashed] (3)--(4);
\draw (4)--(0);
\draw (0)--(5);
\draw[dashed] (5)--(6);
\draw[scale=2,dashed] (6)  to[in=215,out=135,loop] (6);
\end{tikzpicture}
\hspace{14mm}
\begin{tikzpicture}[baseline=0]
\draw (0,0) node[v](1){};
\draw (-0.71, 0.71) node[v](2){};
\draw (-0.71,-0.71) node[v](3){};
\draw ( 0.71, 0.71) node[v](4){};
\draw ( 1.42, 0) node[v](5){};
\draw ( 0.71,-0.71) node[v](6){};
\draw (2)--(1)--(3);
\draw[dashed] (2)--(3);
\draw (4)--(1)--(6)--(5);
\draw[dashed] (4)--(5);
\end{tikzpicture}
\caption{Examples of frame circuits}\label{Fig:exam_frame_circuits}
\end{figure}

The proposition below is a paraphrase of Corollary \ref{circuit nonfree} for signed-graphic arrangements. 
\begin{proposition}\label{frame circuit nonfree}
Let $ G $ be a frame circuit. 
Then $ \mathcal{A}(G) $ is generic. 
Moreover, if $ |E_{G}| \geq 4 $, then $ \mathcal{A}(G) $ is non-free. 
\end{proposition}

In order to determine that a signed-graphic arrangement is non-free, it is important to describe the localizations of a signed-graphic arrangement in terms of signed graphs.

A signed graph $ F=(F^{+},F^{-},L_{F}) $ is said to be a \textbf{subgraph} of a signed graph $ G=(G^{+},G^{-},L_{G}) $ if $ F^{+}, F^{-} $ are subgraphs of $ G^{+},G^{-} $, respectively, on the same vertex set $ V_{F} \subseteq V_{G} $ and $ L_{F} \subseteq L_{G} $. 
The signed-graphic arrangement $ \mathcal{A}(F) $ corresponding to the subgraph $ F $ can be naturally regarded as a subarrangement of $ \mathcal{A}(G) $. 

A subarrangement of an arrangement $ \mathcal{A} $ is a localization if and only if it is a flat of the matroid on $ \mathcal{A} $. 
Hence we have the following proposition: 
\begin{proposition}[{\cite[Proposition 1.4.11(ii)]{oxley2011matroid}}]\label{matroid_flat}
Let $ F $ be a subgraph of a signed graph $ G $. 
The signed-graphic arrangement $ \mathcal{A}(F) $ is a localization of $ \mathcal{A}(G) $ if and only if 
\begin{align*}
\Set{ e \in E_{G} | \text{$ e $ and some edges in $ F $ form a frame circuit} } \subseteq E_{F}. 
\end{align*}
\end{proposition}

\begin{example}
Consider the signed graphs in Figure \ref{Fig:examples non-free}. 
\begin{figure}[t]
\centering
\begin{tikzpicture}
\draw (0,1) node[v](1){};
\draw (0,0) node[v](2){};
\draw (1,0) node[v](3){};
\draw (1,1) node[v](4){};
\draw (1)--(4)--(3)--(1)--(2)--(3);
\draw[bend left, dashed] (2) to (1);
\draw[bend left, dashed] (4) to (3);
\draw (0.5,-0.5) node(){$ G_{1} $};
\end{tikzpicture}
\hspace{10mm}
\begin{tikzpicture}
\draw (0,1) node[v](1){};
\draw (0,0) node[v](2){};
\draw (1,0) node[v](3){};
\draw (1,1) node[v](4){};
\draw (1)--(4);
\draw (2)--(3);
\draw[dashed] (2)--(1);
\draw[dashed] (4)--(3);
\draw (0.5,-0.5) node(){$ F_{1} $};
\end{tikzpicture}
\hspace{10mm}
\begin{tikzpicture}
\draw (0,1) node[v](1){};
\draw (0,0) node[v](2){};
\draw (1,0) node[v](3){};
\draw (1,1) node[v](4){};
\draw (1)--(4)--(3)--(1)--(2)--(3);
\draw[bend left, dashed] (2) to (1);
\draw[bend left, dashed] (1) to (4);
\draw (0.5,-0.5) node(){$ G_{2} $};
\end{tikzpicture}
\hspace{10mm}
\begin{tikzpicture}
\draw (0,1) node[v](1){};
\draw (0,0) node[v](2){};
\draw (1,1) node[v](4){};
\draw (2)--(1)--(4);
\draw[bend left, dashed] (2) to (1);
\draw[bend left, dashed] (1) to (4);
\draw (0.5,-0.5) node(){$ F_{2} $};
\end{tikzpicture}
\hspace{10mm}
\begin{tikzpicture}
\draw (0,0.86) node[v](1){};
\draw (-0.5,0) node[v](2){};
\draw ( 0.5,0) node[v](3){};
\draw (0,-0.5) node(){$ G_{3} $};
\draw (1)--(2)--(3)--(1);
\draw[bend right,dashed] (2) to (3);
\draw[scale=1.5,dashed] (1)  to[in=135,out=45,loop] (1);
\end{tikzpicture}
\caption{Examples of signed graphs corresponding to non-free arrangements}\label{Fig:examples non-free}
\end{figure}
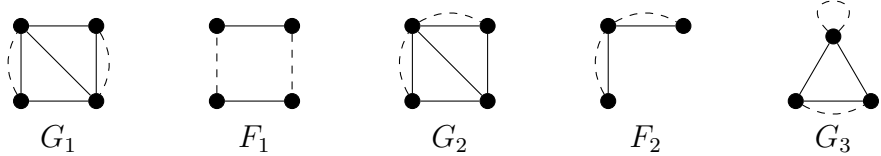
The graph $ G_{1} $ has the balanced cycle $ F_{1} $ and $ G_{2} $ has the tight handcuff $ F_{2} $. 
By Proposition \ref{matroid_flat}, we have that $ \mathcal{A}(F_{1}) $ and $ \mathcal{A}(F_{2}) $ are localizations of $ \mathcal{A}(G_{1}) $ and $ \mathcal{A}(G_{2}) $, respectively. 
Hence, by Propositions \ref{Zaslavsky frame circuit}, \ref{frame circuit nonfree} and \ref{OT_localization}, these four arrangements are non-free. 
A direct computation or the result of Edelman and Reiner (Theorem \ref{ER free}) shows that $ \mathcal{A}(G_{3}) $ is non-free. 
However, every proper localization of $ \mathcal{A}(G_{3}) $ is free. 
Hence it is not enough to focus on circuits when determining freeness of signed graphic arrangements, which is different form the case of graphic arrangements. 
\end{example}

Given a signed graph $ G=(G^{+},G^{-},L_{G}) $, the simple graph $ G^{+} $ can be regarded as a subgraph of $ G $. 
\begin{proposition}\label{G+ is a localization}
Every graphic arrangement $ \mathcal{A}(G^{+}) $ is a localization of a signed-graphic arrangement $ \mathcal{A}(G) $. 
Moreover, if $ \mathcal{A}(G) $ is free, then $ \mathcal{A}(G^{+}) $ is free, or equivalently $ G^{+} $ is chordal.  
\end{proposition}
\begin{proof}
By Propositions \ref{Zaslavsky frame circuit} and \ref{matroid_flat}, we have $ \mathcal{A}(G^{+}) $ is a localization of $ \mathcal{A}(G) $. 
If $ \mathcal{A}(G) $ is free, then $ \mathcal{A}(G^{+}) $ is free by Proposition \ref{OT_localization}. 
Using Theorem \ref{Stanley graphic arrangement}, we have that $ \mathcal{A}(G^{+}) $ is free if and only if $ G $ is chordal. 
\end{proof}

For a signed graph $ G=(G^{+},G^{-},L_{G}) $ and a subset $ W \subseteq V_{G} $, define the \textbf{subgraph induced by $ W $} by $ G[W] \coloneqq (G^{+}[W], G^{-}[W], W \cap L_{G}) $. 
A subgraph $ F $ is called an \textbf{induced subgraph} of $ G $ if $ F=G[W] $ for some $ W \subseteq V_{G} $. 
Using Propositions \ref{Zaslavsky frame circuit} and \ref{matroid_flat}, we obtain the following: 
\begin{proposition}\label{ind_sub is a localization}
Let $ F $ be an induced subgraph of a signed graph $ G $. 
Then $ \mathcal{A}(F) $ is a localization of $ \mathcal{A}(G) $. 
In particular, if $ \mathcal{A}(G) $ is free, then $ \mathcal{A}(F) $ is free. 
\end{proposition} 

Given a signed graph $ G $, we can obtain a new signed graph $G^\nu$ using a \textbf{switching function}
$\nu\colon V_G\to \{\pm 1\}$.
The signed graph $ G^{\nu} $ consists of the following data: 
\begin{itemize}
\item $V_{G^\nu} \coloneqq V_{G}$.
\item $E_{G^\nu}^{+} \coloneqq \Set{\{u,v\}\in E_{G}^{+} | \nu(u)=\nu(v)} \cup \Set{\{u,v\}\in E_{G}^{-} | \nu(u)\ne\nu(v)}$.
\item$E_{G^\nu}^{-} \coloneqq \Set{\{u,v\}\in E_{G}^{+} | \nu(u)\ne\nu(v)} \cup \Set{\{u,v\}\in E_{G}^{-} | \nu(u)=\nu(v)}$.
\item $L_{G^\nu} \coloneqq L_{G}$.
\end{itemize}
Note that a switching affects a signed-graphic arrangement as the coordinate exchange $ x_{i} \mapsto \nu(i)x_{i} $ for each $ i \in V_{G} $. 
Therefore a switching preserves freeness of signed-graphic arrangements. 
Moreover, a switching preserves the balance of cycles. 

Recall that a \textbf{balanced chord} of a balanced cycle is an edge that is not part of the cycle but connects two vertices of the cycle, and moreover that separates the cycle into two balanced cycles. 
A signed graph is called \textbf{balanced chordal} if every balanced cycle of length at least $ 4 $ has a balanced chord. 
Note that balanced chordality does not require any information on the loops. 
\begin{lemma}\label{free => balanced chordal}
Suppose that a signed-graphic arrangement $ \mathcal{A}(G) $ is free. 
Then $ G $ is balanced chordal. 
\end{lemma}
\begin{proof}
Assume that $ G $ is not balanced chordal. 
Then there exists a balanced cycle $ C $ of length at least four with no balanced chords. 
Let $ \nu $ be a switching function on $ G $ such that $ C^{\nu} $ consists of positive edges. 
Since a switching preserves the balance of cycles, we have that $ C^{\nu} $ is a chordless cycle of $ (G^{\nu})^{+} $. 
Therefore $ \mathcal{A}(G^{\nu}) $ is non-free by Proposition \ref{G+ is a localization}. 
Since switching preserves freeness, we have $ \mathcal{A}(G) $ is also non-free. 
Thus, the assertion has been proven. 
\end{proof}

\subsection{Coloring and the chromatic polynomials}
One of remarkable properties of the chromatic polynomial of a simple graph is that the chromatic polynomial coincides with the characteristic polynomial of the corresponding graphic arrangement. 
The notions of coloring and the chromatic polynomial were extended to signed graphs by Zaslavsky \cite{zaslavsky1982signed-dm}.

\begin{definition}
For a positive integer $ k $, let $ \Lambda_{k} \coloneqq \{0, \pm 1, \dots, \pm k \} $. 
A \textbf{proper $ k $-coloring} of a signed graph $ G $ is a function $ \gamma \colon V_{G} \rightarrow \Lambda_{k} $ such that 
\begin{enumerate}[(1)]
\item $ \gamma(u) \neq \gamma(v) $ if $ \{u,v\} \in E_{G}^{+} $. 
\item $ \gamma(u) \neq -\gamma(v) $ if $ \{u,v\} \in E_{G}^{-} $. 
\item $ \gamma(v) \neq 0 $ if $ v \in L_{G} $. 
\end{enumerate}
\end{definition}

\begin{theorem}[Zaslavsky {\cite[Theorem 2.2]{zaslavsky1982signed-dm}}]
There exists a polynomial $ \chi(G,t) $ such that $ \chi(G,2k+1) $ is equal to the number of proper $ k $-coloring of $ G $ for any $k \geq 1$. 
We call $ \chi(G,t) $ the \textbf{chromatic polynomial} of $ G $. 
\end{theorem}
\begin{theorem}[Zaslavsky {\cite[Lemma 4.3]{zaslavsky2012signed-jociss}}]\label{Zaslavsky chromatic characteristic}
The chromatic polynomial of $ G $ and the characteristic polynomial of $ \mathcal{A}(G) $ coincide. 
Namely $ \chi(G,t)=\chi(\mathcal{A}(G),t) $. 
\end{theorem}
Our method to prove the freeness is Abe's division theorem (Theorem \ref{Abe Division Theorem}). 
Thanks to Theorem \ref{Zaslavsky chromatic characteristic}, we can describe the characteristic polynomial of a signed-graphic arrangement in terms of coloring of the corresponding signed graph. 
Recall that $ B_{n} = (K_{n},K_{n},\{1,\dots,n\}) $ denotes the complete signed graph with loops. 
The purpose of the rest of this subsection is to prove Lemma \ref{gluing}, which describes the chromatic polynomial of a signed graph obtained by gluing two signed graphs together along a complete signed graph with loops. 

\begin{proposition}\label{coloring1}
Let $ G $ be a signed graph which contains $ B_{n} $ as a subgraph. 
For two proper $ k $-colorings $ \gamma_{1}, \gamma_{2} $ of $ B_{n} $, let $ C_{1}, C_{2} $ be the set of proper $ k $-colorings of $ G $ whose restrictions to $ B_{n} $ are equal to $ \gamma_{1},\gamma_{2} $, respectively. 
Then $ |C_{1}|=|C_{2}| $.
\end{proposition}
\begin{proof}
For each $ i \in \{1,2\} $, the coloring $ \gamma_{i} $ is injective and the image $ \gamma_{i}(V_{B_{n}}) $ is disjoint to $ -\gamma_{i}(V_{B_{n}}) \coloneqq \Set{-\gamma_{i}(v) | v \in V_{B_{n}}} $. 
We define a bijection $ \phi \colon \gamma_{1}(V_{B_{n}}) \cup (-\gamma_{1}(V_{B_{n}})) \rightarrow \gamma_{2}(V_{B_{n}}) \cup (-\gamma_{2}(V_{B_{n}})) $ by for any $ a \in \gamma_{1}(V_{B_{n}}) $
\begin{align*}
\phi(\pm a) \coloneqq \pm (\gamma_{2} \circ \gamma_{1}^{-1})(a). 
\end{align*}
Then there exists a bijection $ \psi \colon \Lambda_{k} \rightarrow \Lambda_{k} $ such that $ \psi(-a)=-\psi(a) $ and the following diagram commutes. 
\begin{figure}[H]
\centering
\begin{tikzcd}[row sep=tiny]
 & \gamma_{1}(V_{B_{n}}) \cup (-\gamma_{1}(V_{B_{n}})) \ar[r, hook] \ar[dd,"\phi"] & \Lambda_{k} \ar[dd,"\psi"] \\
V_{B_{n}} \ar[ru, end anchor=west, "\gamma_{1}"] \ar[rd, end anchor=west, "\gamma_{2}"'] & & \\
 & \gamma_{2}(V_{B_{n}}) \cup (-\gamma_{2}(V_{B_{n}})) \ar[r, hook] & \Lambda_{k}
\end{tikzcd}
\end{figure}

Define a map $ F \colon C_{1} \rightarrow C_{2} $ by $ F(\gamma) \coloneqq \psi \circ \gamma $. 
We show that $ F $ is well-defined, that is, $ F(\gamma) \in C_{2} $ for any $ \gamma \in C_{1} $. 
First we show that $ F(\gamma) $ is a proper coloring of $ G $. 
This follows by the following arguments. 
\begin{align*}
\{u,v\} \in E_{G}^{+} 
&\Rightarrow \gamma(u) \neq \gamma(v)
\Rightarrow (\psi \circ \gamma)(u) \neq (\psi \circ \gamma)(v) \\
&\Rightarrow F(\gamma)(u) \neq F(\gamma)(v). \\
\{u,v\} \in E_{G}^{-} 
&\Rightarrow \gamma(u) \neq -\gamma(v)
\Rightarrow \psi(\gamma(u)) \neq \psi(-\gamma(v))
\Rightarrow (\psi \circ \gamma)(u) \neq -(\psi \circ \gamma)(v) \\
&\Rightarrow F(\gamma)(u) \neq -F(\gamma)(v). \\
v \in L 
&\Rightarrow \gamma(v) \neq 0
\Rightarrow \psi(\gamma(v)) \neq \psi(0)
\Rightarrow (\psi \circ \gamma)(v) \neq 0 \\
&\Rightarrow F(\gamma)(v) \neq 0. 
\end{align*}
Next we show that the restriction of $ F(\gamma) $ to $ V_{B_{n}} $ equals $ \gamma_{2} $. 
For each vertex $ v \in V_{B_{n}} $
\begin{align*}
F(\gamma)(v) = (\psi \circ \gamma)(v) = \psi(\gamma(v)) = \psi(\gamma_{1}(v)) = \gamma_{2}(v). 
\end{align*}
Hence we may conclude that $ F(\gamma) \in C_{2} $. 

The map $ F $ is injective since $ \psi $ is bijective. 
Therefore we have $ |C_{1}| \leq |C_{2}| $. 
By the similar argument we obtain $ |C_{1}|=|C_{2}| $. 
\end{proof}
\begin{proposition}\label{coloring2}
Let $ G $ be a signed graph which contains $ B_{n} $ as a subgraph. 
For a proper $ k $-coloring $ \gamma $ of $ B_{n} $, let $ C $ denote the set of proper $ k $-colorings of $ G $ whose restrictions to $ B_{n} $ are equal to $ \gamma $. 
Then 
\begin{align*}
|C|=\frac{\chi(G,2k+1)}{\chi(B_{n},2k+1)}. 
\end{align*}
\end{proposition}
\begin{proof}
Let $ \gamma_{1} \coloneqq \gamma, \gamma_{2}, \dots, \gamma_{m} $ be the proper $ k $-colorings of $ B_{n} $, where $ m = \chi(B_{n},2k+1) $. 
For each $ i \in [m] $, let $ C_{i} $ denote the set of proper $ k $-colorings of $ G $ whose restrictions to $ B_{n} $ are equal to $ \gamma_{i} $. 
Then we have that $ \chi(G,2k+1)=\sum_{i=1}^{m}|C_{i}| $. 
By Proposition \ref{coloring1}, the cardinalities of $ |C_{i}| $ are equal to each other. 
Hence $ \chi(G,2k+1)=m|C|=\chi(B_{n},2k+1)|C| $. 
Thus the assertion holds. 
\end{proof}

\begin{lemma}\label{gluing}
Let $ G $ be a signed graph obtained by gluing two signed graphs $ G_{1}, G_{2} $ along a complete signed graph with loops $ B_{n} $. 
Namely, $ G_{1} $ and $ G_{2} $ are induced subgraphs of $ G $ such that $ G_{1} \cup G_{2} = G $ and $ G_{1} \cap G_{2} = B_{n} $. 
Then the following hold:
\begin{enumerate}[(1)]
\item \label{gluing 1}
${\displaystyle
\chi(G,t) = \frac{\chi(G_{1},t)\chi(G_{2},t)}{\chi(B_{n},t)}.}$
\item \label{gluing 2}
If $ G_{1} $ and $ G_{2} $ are balanced chordal, then $ G $ is also balanced chordal. 
\end{enumerate}
\end{lemma}
\begin{proof}
(\ref{gluing 1})
For each proper $ k $-coloring of $ G_{1} $, the number of possible proper $ k $-colorings of $ G_{2} $ is $ \chi(G_{2},2k+1)/\chi(B_{n},2k+1) $ by Proposition \ref{coloring2}. 
Hence we have 
\begin{align*}
\chi(G,2k+1) = \frac{\chi(G_{1},2k+1) \chi(G_{2},2k+1)}{\chi(B_{n},2k+1)}. 
\end{align*}
Since this equality holds for any positive integer $ k $, the assertion holds. 

(\ref{gluing 2})
Suppose that $ G_{1} $ and $ G_{2} $ are balanced chordal. 
Take a balanced cycle $ C $ of length at least $ 4 $. 
If $ C \subseteq V_{G_{1}} $ or $ C \subseteq V_{G_{2}} $, then $ C $ has a balanced chord since $ G_{1},G_{2} $ are balanced chordal. 
Now, suppose that $ C $ has vertices in both of $ V_{G_{1}}\setminus V_{B_{n}} $ and $ V_{G_{2}}\setminus V_{B_{n}} $. 
Then $ C $ has two non-consecutive vertices in $ B_{n} $. 
Either of the positive or the negative edge between these vertices is a balanced chord of $ C $. 
Therefore $ G $ is also balanced chordal. 
\end{proof}

\subsection{Edge contractions}
We define an edge contraction of a signed graph $ G $ so that the corresponding signed-graphic arrangement becomes the restriction with respect to the hyperplane corresponding to the edge. 

Let $ \{v,w\} $ be a positive edge of $ G $. 
We may define the contraction of the edge in the same way as for simple graphs. 
However, in order to clarify the notion, we fix a direction $ (v,w) $ of the edge $ \{v,w\} $. 
The contraction $ G/(v,w) $ consists of the following data:
\begin{itemize}
\item $ V_{G/(v,w)} \coloneqq V_{G} \setminus \{v\} $. 
\item $ E_{G/(v,w)}^{+} \coloneqq E_{G \setminus \{v\}}^{+} \cup \Set{\{u,w\} | \{u,v\} \in E_{G}^{+}} $. 
\item $ E_{G/(v,w)}^{-} \coloneqq E_{G \setminus \{v\}}^{-} \cup \Set{\{u,w\} | \{u,v\} \in E_{G}^{-}} $. 
\item $ L_{G/(v,w)} \coloneqq 
\begin{cases}
L_{G\setminus \{v\}} \cup \{w\} & \text{ if } v \in L_{G} \text{ or } \{v,w\} \in E_{G}^{-}, \\
L_{G\setminus\{v\}} & \text{ otherwise. }
\end{cases} $
\end{itemize}
It is easy to see that the resulting graph is independent of choice of a direction. 
Namely, $ G/(v,w) $ and $ G/(w,v) $ are isomorphic. 

For a negative edge $ \{v,w\} $, the contraction $ G/(v,w) $ consists of the following data:
\begin{itemize}
\item $ V_{G/(v,w)} \coloneqq V_{G} \setminus \{v\} $. 
\item $ E_{G/(v,w)}^{+} \coloneqq E_{G \setminus \{v\}}^{+} \cup \Set{\{u,w\} | \{u,v\} \in E_{G}^{-}} $. 
\item $ E_{G/(v,w)}^{-} \coloneqq E_{G \setminus \{v\}}^{-} \cup \Set{\{u,w\} | \{u,v\} \in E_{G}^{+}} $. 
\item $ L_{G/(v,w)} \coloneqq 
\begin{cases}
L_{G\setminus \{v\}} \cup \{w\} & \text{ if } v \in L_{G} \text{ or } \{v,w\} \in E_{G}^{+}, \\
L_{G\setminus\{v\}} & \text{ otherwise. }
\end{cases} $
\end{itemize}
In this case, the contractions $ G/(v,w) $ and $ G/(w,v) $ are not isomorphic but switching equivalent, that is, there exists a switching function $ \nu $ on $ G/(v,w) $ such that $ (G/(v,w))^{\nu} $ and $ G/(w,v) $ are isomorphic 
(see, for example, \cite[Lemma 2.8]{zaslavsky2012signed-jociss}). 

We can define the contraction for a loop. 
However, it is not required in this paper. 
From the construction of the contraction, we obtain the following proposition: 
\begin{proposition}\label{contraction restriction}
Let $ \{v,w\} $ be a positive or negative edge of a signed graph $ G $ and $ H $ the corresponding hyperplane in $ \mathcal{A}(G) $. 
Then $ \mathcal{A}(G/(v,w)) = \mathcal{A}(G)^{H} $. 
\end{proposition}

\subsection{Simplicial extensions}
As mentioned in Theorems \ref{Stanley graphic arrangement} and \ref{FG}, the freeness of graphic arrangements is characterized by existence of a perfect elimination ordering.
In this section, we extend the concept and introduce signed-simplicial vertices and signed elimination orderings.

A signed graph $ G $ is called \textbf{balanced} if every cycle of $ G $ is balanced. 
Note that there is no relationship between balanced chordality and being balanced. 
Define the \textbf{rank} of $ G $ by $ \rank(G) \coloneqq |V_{G}|-b(G) $, where $ b(G) $ denotes the number of balanced connected components of $ G $. 
\begin{theorem}[Zaslavsky {\cite[Theorem 3.5]{zaslavsky2012signed-jociss}}]\label{Zaslavsky rank}
Given a signed graph $ G $, we have $ \rank(\mathcal{A}(G)) = \rank(G) $. 
\end{theorem}

\begin{definition}
A vertex $ v $ of a signed graph $ G $ is called \textbf{signed simplicial} if the following statements hold: 
\begin{enumerate}[(1)]
\item If $ \{u_{1},v\}, \{u_{2},v\} \in E_{G}^{+} $ or $ \{u_{1},v\}, \{u_{2},v\} \in E_{G}^{-} $ then $ \{u_{1},u_{2}\} \in E_{G}^{+} $. 
\item If $ \{u_{1},v\} \in E_{G}^{+} $ and $ \{u_{2},v\} \in E_{G}^{-} $ then $ \{u_{1},u_{2}\} \in E_{G}^{-} $. 
\item If $ \{u,v\} \in E_{G}^{+} \cup E_{G}^{-} $ with $ v \in L_{G} $, or $ \{u,v\} \in E_{G}^{+} \cap E_{G}^{-} $ then $ u \in L_{G} $. 
\end{enumerate}
\end{definition}

Adding a signed simplicial vertex affects signed graphs as follows: 
\begin{proposition}\label{add ss}
Let $ v $ be a signed-simplicial vertex of a signed graph $ G $ and let $ F \coloneqq G \setminus \{v\} $. 
Then the following hold. 
\begin{enumerate}[(1)]
\item \label{add ss 1} Suppose that $ v $ has an adjacent vertex $ w $. 
Then $ G/(v,w) = F $. 
\item  \label{add ss 2} If $ v $ has a loop or an adjacent vertex, then $ \rank(G) = \rank(F)+1 $. 
\item \label{add ss 3} Let $ d $ denote the degree of $ v $, that is, the number of incident edges and a loop of $ v $.
Then $ \chi(G,t) = (t-d)\chi(F,t)$.
\item \label{add ss 4} $ G $ is balanced chordal if and only if $ F $ is balanced chordal. 
\end{enumerate}
\end{proposition}
\begin{proof}
(\ref{add ss 1})
We need to show that $ E_{G/(v,w)}^{+} = E_{F}^{+}, E_{G/(v,w)}^{-} = E_{F}^{-} $, and $ L_{G/(v,w)} = L_{F} $. 
Suppose that $ \{v,w\} $ is a positive edge. 
If $ \{u,v\} \in E_{G}^{+} $, then we have $ \{ u,w \} \in E_{G}^{+} $. 
Hence $ \Set{\{u,w\} | \{u,v\} \in E_{G}^{+}} \subseteq E_{F}^{+} $. 
Therefore $ E_{G/(v,w)}^{+} = E_{F}^{+} $. 
The other cases are similar. 

(\ref{add ss 2}) 
Since $ \rank(G)=|V_{G}|-b(G) $ and $ \rank(F)=|V_{G}|-1-b(F) $, it suffices to show that $ b(G)=b(F) $. 
When $ v $ is isolated and has a loop, the equation $ b(G)=b(F) $ holds since the loop graph ($ 1 $-cycle) is unbalanced. 
From now on, we assume that $ v $ has an adjacent vertex. 
Without loss of generality, we may assume that $ G $ is connected since removing the vertex $ v $ affects only the connected component of $ G $ including $ v $. 
In this case, $ F $ is also connected since $ v $ is simplicial. 
Therefore it suffices to show that $ F $ is balanced if and only if $ G $ is balanced. 

When $ G $ is balanced, $ F $ is balanced trivially. 
To show the converse, suppose that $ F $ is balanced.  
Let $ C $ be a cycle of $ G $ containing $ v $. 
It is sufficient to show that $ C $ is balanced. 
If the length of $ C $ is $ 1 $ or $ 2 $, then $ F $ has a loop since $ v $ is signed simplicial, which is a contradiction. 
Hence the length of $ C $ is at least $ 3 $. 
Write $ C $ as a sequence of vertices $ v=v_{1},v_{2}, \dots, v_{k} $. 
Assume that $ C $ is unbalanced. 
Since $ v $ is signed simplicial, there is an edge $ e=\{v_{2},v_{k}\} $ forming a balanced $ 3 $-cycle together with the edges $ \{v_{k},v_{1}\}, \{v_{1},v_{2}\} $ in the cycle $ C $. 
Hence the cycle in $ F $ consisting of $ e $ and the edges $ \{v_{2},v_{3}\}, \dots, \{v_{k-1},v_{k}\} $ in the cycle $ C $ is unbalanced, which is a contradiction. 
As a result, $ C $ is balanced. 

(\ref{add ss 3})
Let $ \gamma $ be a proper $ k $-coloring on $ G\setminus\{v\} $, where $ k $ is sufficiently large. 
It is sufficient to show that the number of proper $ k $-colorings on $ G $ which are extensions of $ \gamma $ is $ (2k+1)-d $. 

First, suppose that $ v $ has no loop. 
By the definition of a proper coloring, the color of $ v $ cannot belong to the following set:
\begin{align*}
\Set{\gamma(w) | \{v,w\} \in E_{G}^{+}} \cup \Set{-\gamma(w) | \{v,w\} \in E_{G}^{-}}. 
\end{align*}
We show that the cardinality of this set coincides with the degree of $ v $, that is, the colors corresponding to the incident edges of $ v $ are different from each other. 

Suppose that $ \{v,w_{1}\}, \{v,w_{2}\} \in E_{G}^{+} $, and $ w_{1} \neq w_{2} $. 
Since $ v $ is signed simplicial, we have $ \{w_{1},w_{2}\} \in E_{G}^{+} $ and hence $ \gamma(w_{1}) \neq \gamma(w_{2}) $. 
When $ \{v,w_{1}\}, \{v,w_{2}\} \in E_{G}^{-} $ with $ w_{1} \neq w_{2} $ or $ \{v,w_{1}\} \in E_{G}^{+}, \{v, w_{2}\} \in E_{G}^{-} $ with $ w_{1} \neq w_{2} $, one can prove the assertion in the same way. 
Next, assume that $ \{v,w\} \in E_{G}^{+} $ and $ \{v,w\} \in E_{G}^{-} $. 
Then we have $ w \in L_{G} $ by the definition of a signed-simplicial vertex. 
Hence $ \gamma(w) \neq 0 $. 
In other words, $ \gamma(w) \neq -\gamma(w) $. 

Now, we consider the case $ v $ has a loop. 
In this case, the color of $ v $ is neither a member of the set above nor $ 0 $. 
Since $ v $ is signed simplicial, every adjacent vertex of $ v $ admits a loop, and hence $0$ does not belong to the set above. 
Therefore the forbidden colors of $ v $ coincides with the degree of $ v $. 

(\ref{add ss 4}) 
If $ G $ is balanced chordal, then $ F $ is also balanced chordal, since $ F $ is an induced subgraph of $ G $. 
In order to prove the converse, suppose that $ F $ is balanced chordal. 
Let $ C $ be a balanced cycle of $ G $ which is of length at least $ 4 $ and contains $ v $. 
Write $ C $ as a sequence of vertices $ v=v_{1}, v_{2}, \dots, v_{k} $. 
Since $ v $ is signed simplicial, we have $ \{v_{2},v_{k}\} $ is an edge and the $ 3 $-cycle $ v_{1},v_{2},v_{k} $ is balanced. 
Hence the $ (k-1) $-cycle $ v_{2},\dots, v_{k} $ is also balanced. 
Therefore the edge $ \{v_{2},v_{k}\} $ is a balanced chord of $ C $.  
\end{proof}

\begin{definition}
An ordering $ (v_{1}, \dots, v_{\ell}) $ of the vertices of a signed graph $ G $ is said to be a \textbf{singed elimination ordering} if $ v_{k} $ is signed simplicial in $ G[\{v_{1}, \dots, v_{k}\}] $ for each $ k \in \{1,\dots, \ell \} $. 
\end{definition}
Let $ G $ be a signed graph and $ F $ an induced subgraph. 
We say that $ G $ is a \textbf{simplicial extension} of $ F $ 
if there exists an ordering $(v_{1}, \ldots ,v_{m})$ of the vertices
in $V_{G} \setminus V_{F}$ such that each $v_{i}$ is signed simplicial
in $G[V_{F} \cup \{ v_{1}, \ldots ,v_{i} \}]$.
Zaslavsky characterized supersolvability of signed-graphic arrangements as follows: 
\begin{theorem}[{Zaslavsky \cite[Theorem 2.2]{zaslavsky2001supersolvable-ejoc}}]\label{Zaslavsky SS}
A signed-graphic arrangement $ \mathcal{A}(G) $ is supersolvable if and only if one of the following conditions holds: 
\begin{enumerate}[(1)]
\item $ G $ has a signed elimination ordering. 
\item $ G $ is a simplicial extension of one of the following:
\begin{enumerate}[(i)]
\item The graph $ D_{3} $ shown in Figure \ref{Fig:D3}. 
\item A signed graph $ F=(F^{+},F^{-},\varnothing) $ in which all edges in $ F^{-} $ are incident to a single vertex $ v $, the set of neighbors of $ v $ in $ F^{-} $ induces a complete subgraph in $ F^{+} $, and $ F^{+} $ has a perfect elimination ordering. 
\end{enumerate}
\end{enumerate}
\end{theorem}
\begin{figure}[t]
\centering
\begin{tikzpicture}
\draw (0,0.865) node[v](1){};
\draw (-0.5,0) node[v](2){};
\draw ( 0.5,0) node[v](3){};
\draw (1)--(2)--(3)--(1);
\draw[bend right, dashed] (1) to (2);
\draw[bend right, dashed] (2) to (3);
\draw[bend right, dashed] (3) to (1);
\end{tikzpicture}
\caption{The graph $ D_{3} $}\label{Fig:D3}
\end{figure}
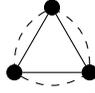
By Theorems \ref{Zaslavsky SS}, \ref{JT SS=>free}, and Lemma \ref{free => balanced chordal}, we obtain the following implications: 
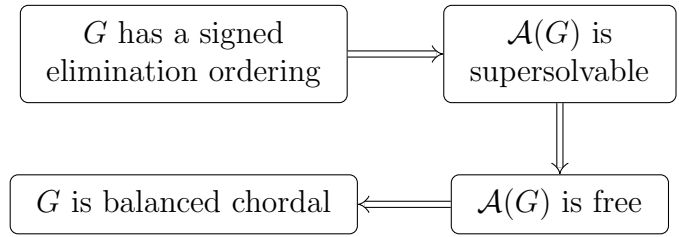
\begin{figure}[H]
\centering
\begin{tikzpicture}
\draw (0,0) node[draw,rectangle, rounded corners=3pt](1){\begin{tabular}{c}
$ G $ has a signed \\ elimination ordering
\end{tabular}};
\draw (5,0) node[draw,rectangle, rounded corners=3pt](2){\begin{tabular}{c}
$ \mathcal{A}(G) $ is \\ supersolvable
\end{tabular}};
\draw (5,-2) node[draw,rectangle, rounded corners=3pt](3){\begin{tabular}{c}
$ \mathcal{A}(G) $ is free
\end{tabular}};
\draw (0,-2) node[draw,rectangle, rounded corners=3pt](4){\begin{tabular}{c}
$ G $ is balanced chordal
\end{tabular}};
\draw[-implies,double equal sign distance] (1)--(2);
\draw[-implies,double equal sign distance] (2)--(3);
\draw[-implies,double equal sign distance] (3)--(4);
\end{tikzpicture}
\end{figure}
However, in contrast to graphic arrangements, these four conditions are not equivalent to each other (see \cite[Lemma 4.5(c)]{edelman1994free}).
These conditions are related to simplicial extension as follows: 
\begin{proposition}\label{ss extension}
Let $ G $ be a signed simplicial extension of a signed graph $ F $. 
Then the following hold: 
\begin{enumerate}[(1)]
\item \label{ss extension1} $ G $ has a signed elimination ordering if and only if $ F $ has a signed elimination ordering. 
\item \label{ss extension2} $ \mathcal{A}(G) $ is supersolvable if and only if $ \mathcal{A}(F) $ is supersolvable. 
\item \label{ss extension3} $ \mathcal{A}(G) $ is free if and only if $ \mathcal{A}(F) $ is free. 
\item \label{ss extension4} $ G $ is balanced chordal if and only if $ F $ is balanced chordal. 
\end{enumerate}
\end{proposition}
\begin{proof}
Without loss of generality, we may assume that $ F=G\setminus\{v\} $, where $ v $ is a signed-simplicial vertex. 

(\ref{ss extension1}) 
This is clear from the definition of a signed elimination ordering. 

(\ref{ss extension2})
If $ \mathcal{A}(G) $ is supersolvable, then $ \mathcal{A}(F) $ is supersolvable by Propositions \ref{ind_sub is a localization} and \ref{Stanley SS lattice}. 
To show the converse, suppose that $ \mathcal{A}(F) $ is supersolvable. 
If $ v $ has no loop and isolated, then $ \mathcal{A}(F) = \mathcal{A}(G) $, and hence $ \mathcal{A}(G) $ is supersolvable. 
Suppose that $ v $ has a loop or an adjacent vertex. 
By Theorem \ref{Zaslavsky rank} and Proposition \ref{add ss}(\ref{add ss 2}), we have $ \rank(\mathcal{A}(G))=\rank(\mathcal{A}(F))+1 $. 

We show that for distinct hyperplanes $ H,H^{\prime} \in \mathcal{A}(G)\setminus\mathcal{A}(F) $, there exists $ H^{\prime\prime} \in \mathcal{A}(F) $ such that $ H \cap H^{\prime} \subseteq H^{\prime\prime} $. 
Suppose that $ H,H^{\prime} $ correspond to positive edges $ \{u_{1},v\},\{u_{2},v\} \in E_{G}^{+} $. 
Since $ v $ is signed simplicial, we have $ \{u_{1},u_{2}\} \in E_{G}^{+} $. 
Let $ H^{\prime\prime} $ be the hyperplane corresponding to the edge $ \{u_{1},u_{2}\} $. 
Then we have $ H \cap H^{\prime} \subseteq H^{\prime\prime} \in \mathcal{A}(F) $. 
The other cases are similar. 
Using Theorem \ref{BEZ}, we conclude that $ \mathcal{A}(G) $ is supersolvable. 

(\ref{ss extension3}) 
If $ \mathcal{A}(G) $ is free, then $ \mathcal{A}(F) $ is free by Propositions \ref{ind_sub is a localization} and \ref{OT_localization}. 
We prove the converse. 
If $ v $ is isolated, then $ \mathcal{A}(G) $ is a product of $ \mathcal{A}(F) $ and a $ 1 $-dimensional arrangement. 
Hence $ \mathcal{A}(G) $ is free. 
Assume that $ v $ has an adjacent vertex $ w $. 
Let $ H $ be the hyperplane corresponding to $ \{v,w\} $. 
Then we have 
\begin{align*}
\mathcal{A}(G)^{H} = \mathcal{A}(G/(v,w)) = \mathcal{A}(F)
\end{align*}
by Propositions \ref{contraction restriction} and \ref{add ss}(\ref{add ss 1}). 
Using Proposition \ref{add ss}(\ref{add ss 3}), and Theorems \ref{Zaslavsky chromatic characteristic} and \ref{Abe Division Theorem}, we have that $ \mathcal{A}(G) $ is free. 

(\ref{ss extension4}) 
This follows by Proposition \ref{add ss}(\ref{add ss 4}). 
\end{proof}

\section{Divisional edges and  vertices}\label{sec:divisional vertex}
In this section, we introduce the notions of divisional edges and vertices of a signed graph. 
These will play an important role to prove our main theorem. 

\begin{definition}
An edge $ \{v,w\} $ of a signed graph $ G $ is said to be \textbf{divisional} if $ \chi(G/(v,w),t) $ divides $ \chi(G,t) $. 
Note that, since $ G/(v,w) $ and $ G/(w,v) $ are switching equivalent and hence the chromatic polynomials of them coincide, the definition of a divisional edge is independent on the choice of directions. 
The endvertices $ v,w $ are called \textbf{divisional vertices}. 
\end{definition}

A motivation of introducing a divisional edge is the lemma below. 
\begin{lemma}\label{division theorem for signed graph}
Let $ e $ be a divisional edge of a signed graph $ G $. 
Suppose that $ \mathcal{A}(G/e) $ is free. 
Then $ \mathcal{A}(G) $ is free. 
\end{lemma}
\begin{proof}
This follows by Proposition \ref{contraction restriction} and Theorem \ref{Abe Division Theorem}. 
\end{proof}

To use this lemma, we need to know what edges are divisional. 
The following proposition raises a sufficient condition. 
\begin{proposition}\label{ss => divisional}
Every incident edge of a signed-simplicial vertex is divisional. 
In other words, a signed-simplicial vertex and its neighbors are divisional. 
\end{proposition}
\begin{proof}
This follows by Proposition \ref{add ss} (\ref{add ss 1}) and (\ref{add ss 3}). 
\end{proof}

\section{Proof of Theorem \ref{main theorem}}\label{sec:proof}
In this section, we focus on the signed-graphic arrangement $ \mathcal{A}(G) $ with $ G^{+} \supseteq G^{-} $ and prove Theorem \ref{main theorem}. 

\subsection{The case $ G^{+} $ is complete}
As mentioned in Section \ref{sec:introduction}, Edelman and Reiner characterized freeness of the signed-graphic arrangements corresponding a signed graph $ G=(K_{\ell},G^{-},L_{G}) $. 
See Subsection \ref{subsec:threshold graph} for terminologies of threshold graphs. 
\begin{theorem}[Edelman-Reiner {\cite[Theorem 4.6]{edelman1994free}}]\label{ER free}
A signed-graphic arrangement $ \mathcal{A}(K_{\ell},G^{-},L_{G}) $ is free if and only if $ G^{-} $ is threshold and $ L_{G} $ is an initial segment of some degree ordering of $ G^{-} $. 
\end{theorem}
When $ \mathcal{A}(K_{\ell},G^{-},L_{G}) $ is free, the signed graph $ (K_{\ell},G^{-},L_{G}) $ is balanced chordal by Lemma \ref{free => balanced chordal}. 
Actually, the following proposition holds. 
\begin{proposition}\label{bc threshold}
Let $ G=(K_{\ell},G^{-}) $. 
Then $ G $ is balanced chordal if and only if $ G^{-} $ is threshold. 
\end{proposition}
\begin{proof}
Assume that $ G^{-} $ is not threshold.
By Theorem \ref{Golumbic}, the signed graph $ G $ has one of left three graphs in Figure \ref{Fig:bc-threshold} as an induced subgraph. 
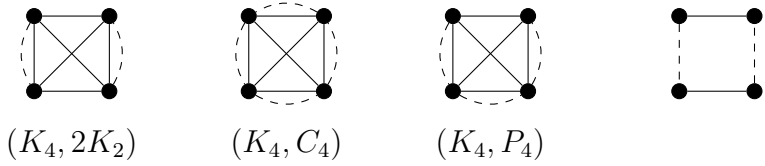
\begin{figure}[t]
\centering
\begin{tikzpicture}[baseline=0]
\draw (0,1) node[v](1){};
\draw (0,0) node[v](2){};
\draw (1,0) node[v](3){};
\draw (1,1) node[v](4){};
\draw (0.5,-0.7) node[]{$ (K_{4},2K_{2}) $};
\draw[] (1)--(2)--(3)--(4)--(1);
\draw[] (1)--(3);
\draw[] (2)--(4);
\draw[bend right,dashed] (1) to (2);
\draw[bend right,dashed] (3) to (4);
\end{tikzpicture} \qquad
\begin{tikzpicture}[baseline=0]
\draw (0,1) node[v](1){};
\draw (0,0) node[v](2){};
\draw (1,0) node[v](3){};
\draw (1,1) node[v](4){};
\draw (0.5,-0.7) node[]{$ (K_{4},C_{4}) $};
\draw[] (1)--(2)--(3)--(4)--(1);
\draw[] (1)--(3);
\draw[] (2)--(4);
\draw[bend right,dashed] (1) to (2);
\draw[bend right,dashed] (2) to (3);
\draw[bend right,dashed] (3) to (4);
\draw[bend right,dashed] (4) to (1);
\end{tikzpicture} \qquad
\begin{tikzpicture}[baseline=0]
\draw (0,1) node[v](1){};
\draw (0,0) node[v](2){};
\draw (1,0) node[v](3){};
\draw (1,1) node[v](4){};
\draw (0.5,-0.7) node[]{$ (K_{4},P_{4}) $};
\draw[] (1)--(2)--(3)--(4)--(1);
\draw[] (1)--(3);
\draw[] (2)--(4);
\draw[bend right,dashed] (1) to (2);
\draw[bend right,dashed] (2) to (3);
\draw[bend right,dashed] (3) to (4);
\end{tikzpicture}\hspace{14mm}
\begin{tikzpicture}[baseline=0]
\draw (0,1) node[v](1){};
\draw (0,0) node[v](2){};
\draw (1,0) node[v](3){};
\draw (1,1) node[v](4){};
\draw (1)--(4);
\draw (2)--(3);
\draw[dashed] (1)--(2);
\draw[dashed] (3)--(4);
\end{tikzpicture}
\caption{Forbidden induced subgraphs and their balanced cycle}\label{Fig:bc-threshold}
\end{figure} 
Each of these contains the balanced cycle of the rightmost graph in Figure \ref{Fig:bc-threshold}, 
which has no balanced chords. 
Hence $ G $ is not balanced chordal. 

Suppose that $ G^{-} $ is threshold. 
Take a balanced cycle $ C $ of length at least $ 4 $. 
If $ C $ contains edges $ \{u,v\},\{v,w\} \in E_{G}^{+} $ or $ \{u,v\},\{v,w\} \in E_{G}^{-} $, then there is a balanced chord $ \{u,w\} \in E_{G}^{+} $ since $ G^{+}=K_{\ell} $ is complete. 
Therefore we may assume that $ C $ is formed by alternating signed edges. 
Let $ \{v_{1},v_{2}\}, \{v_{2},v_{3}\}, \{v_{3},v_{4}\} $ be consecutive edges of $ C $, where $ \{v_{1},v_{2}\}, \{v_{3},v_{4}\} \in E_{G}^{-} $ and $ \{v_{2},v_{3}\} \in E_{G}^{+} $. 
Put $ F=G[\{v_{1},v_{2},v_{3},v_{4}\}] $. 
Then $ F^{-} $ is threshold since $ G^{-} $ is threshold. 
There is no isolated vertices in $ F^{-} $. 
Hence $ F^{-} $ has a dominating vertex. 
Therefore $ \{v_{1}, v_{3}\} \in E_{G}^{-} $ or $ \{v_{2},v_{4}\} \in E_{G}^{-} $, which is a balanced chord of $ C $. 
Thus $ G $ is balanced chordal. 
\end{proof}

The following proposition is a translation of a result of Edelman and Reiner using terminologies from signed graphs. 
\begin{proposition}[Edelman-Reiner {\cite[Lemma 4.7]{edelman1994free}}]\label{ER SEO}
Let $ G=(K_{\ell},G^{-},L_{G}) $ be a signed graph with $ G^{-} $ threshold. 
Suppose that $ L_{G} $ is an initial segment of some degree order $ (v_{1}, \dots, v_{\ell}) $ of $ G^{-} $ and that $ L_{G} $ contains at least one endvertex from every edge of $ G^{-} $. 
Then $ (v_{1}, \dots, v_{\ell}) $ is a signed elimination ordering. 
\end{proposition}
\begin{corollary}\label{deg minimal ss}
Let $ G=(K_{\ell},G^{-},V_{G}) $ with $ G^{-} $ threshold. 
Then a vertex $ v $ such that $ \deg_{G^{-}}(v) $ is minimal is signed simplicial. 
In particular, every vertex of $ G $ is divisional. 
\end{corollary}
\begin{proof}
This follows by Propositions \ref{ER SEO} and \ref{ss => divisional}. 
\end{proof}

When $ G^{+} $ is complete and every vertex has a loop, the situation is almost as simple as in the case of simple graphs. 
\begin{proposition}\label{G+ comp and full loops}
Let $ G=(K_{\ell},G^{-},V_{G}) $. 
The following are equivalent. 
\begin{enumerate}[(1)]
\item $ G $ is balanced chordal. 
\item $ G^{-} $ is threshold. 
\item $ G $ has a signed elimination ordering. 
\item $ \mathcal{A}(G) $ is supersolvable. 
\item $ \mathcal{A}(G) $ is free.  
\end{enumerate}
\end{proposition}
\begin{proof}
$ (1) \Leftrightarrow (2) $ By Proposition \ref{bc threshold}.
 
$ (2) \Rightarrow (3) $ By Proposition \ref{ER SEO}. 

$ (3) \Rightarrow (4) $ By Theorem \ref{Zaslavsky SS}. 

$ (4) \Rightarrow (5) $ By Theorem \ref{JT SS=>free}.  

$ (5) \Rightarrow (1) $ By Lemma \ref{free => balanced chordal}.  
\end{proof}

\subsection{Lemmas}
In this subsection, we assume that a signed graph $ G $ satisfies the following conditions: 
\begin{itemize}
\item $ G^{+} \supseteq G^{-} $. 
\item $ G^{+} $ is non-complete. 
\item $ G $ is balanced chordal, and hence $ G^{+} $ is chordal. 
\item $ L_{G}=V_{G} $, that is, every vertex admits a loop. 
\end{itemize}

\begin{lemma}\label{negative chord}
Take a minimal vertex separator $ S $ of $ G^{+} $ and distinct vertices $ u,v \in S $. 
Let $ A,B $ be vertex sets of distinct connected components of $ G \setminus S $. 
Assume that there exists a cycle $ C $ satisfying the following properties:
\begin{enumerate}[(i)]
\item \label{negative chord 1} $ V_{C} \cap S = \{u,v\} $. 
\item \label{negative chord 2} $ C $ has exactly two negative edges $ e,e^{\prime} $. 
\item \label{negative chord 3} One of the endvertices of $ e,e^{\prime} $ belongs to $ A,B $, respectively. 
\end{enumerate}
Then we have $ \{u,v\} \in E_{G}^{-} $. 
\end{lemma}
\begin{proof}
Assume that $ C $ is a cycle of minimal length satisfying the conditions. 
Clearly $ C $ is a balanced cycle of length at least four. 
Hence $ C $ must have a balanced chord. 
The minimality of $C$ implies $ V_{C} \subseteq A \cup S \cup B $ since $ S $ is a clique of $ G^{+} $ by Theorem \ref{Dirac minmal vertex separator}. 
There are no edges between a vertex in $ A $ and a vertex in $ B $ since $ S $ is a minimal vertex separator of $ G^{+} $ and $ G^{+} \supseteq G^{-} $. 
The minimality of $ C $ implies that $ C $ has no balanced chords between a vertex in $ A $ and another vertex in $ A \cup S $. 
Similarly, $ C $ has no balanced chords between a vertex in $ B $ and a vertex in $ S \cup B $. 
Hence only the edge $ \{u,v\} $ can be a balanced chord. 
The positive edge $ \{u,v\} \in E_{G}^{+} $ is not a balanced chord of $ C $. 
Therefore we have $ \{u,v\} \in E_{G}^{-} $. 
\end{proof}

Let $ \mathcal{G} $ be the clique-separator graph of $ G^{+} $ 
(See Subsection \ref{subsec:clique-separator graph} for results about clique-separator graphs). 
Every sink box has at least two clique nodes since $ G^{+} $ is non-complete. 
Let $ P $ be a path in a sink box from a clique node to another clique node. 
We write $ P $ as follows. 
\begin{figure}[H]
\centering
\begin{tikzpicture}
\draw (0,0) node[](C1){$ C_{1} $}; 
\draw (1,0) node[](S2){$ S_{2} $};
\draw (2,0) node[](C2){$ C_{2} $};
\draw (4.6,0) node[](Sk){$ S_{k} $};
\draw (5.6,0) node[](Ck){$ C_{k} $};
\draw (C1)--(S2)--(C2)--(2.8,0);
\draw[dotted] (2.8,0)--(3.8,0);
\draw (3.8,0)--(Sk)--(Ck); 
\end{tikzpicture}
\end{figure}
Take vertices $ a \in C_{1} \setminus S_{2} $ and $ b \in C_{k}\setminus S_{k} $. 
By Corollary \ref{Ibarra 1 cor}, the set of $ (a,b) $-minimal separators of $ G^{+} $ coincides with $ \{S_{2}, \dots, S_{k}\} $. 
Let $ S_{i} $ be a minimal $ (a,b) $-separator whose cardinality is minimal. 
For $v \in V_{G}$, let $ N_{G^{-}[A]}(v) $ denote the set of vertices in a subset $ A \subseteq V_{G} $ connected to $ v $ by an negative edge. 
\begin{lemma}\label{separator comp signed}
Suppose that $ N_{G^{-}[C_{1}]}(a) \not \subseteq S_{i} $ and $ N_{G^{-}[C_{k}]}(b) \not \subseteq S_{i} $. 
Then $ G[S_{i}] $ is a complete signed graph with loops. 
\end{lemma}
\begin{proof}
From the assumptions, there exist vertices $ a^{\prime} \in C_{1}\setminus S_{i}, b^{\prime} \in C_{k}\setminus S_{i} $ such that $ \{a,a^{\prime}\} \in E_{G}^{-} $ and $ \{b,b^{\prime}\} \in E_{G}^{-} $. 
Moreover, by our hypothesis, every vertex has a loop and $ G^{+}[S_{i}] $ is a complete simple graph from Theorem \ref{Dirac minmal vertex separator}. 
Therefore we only need to show that $ G^{-}[S_{i}] $ is also a complete simple graph. 
Take distinct vertices $ u,v \in S_{i} $. 
By Corollary \ref{Menger cor} and Theorem \ref{Ibarra 1}, we obtain a cycle $ C $ of $ G^{+}[P] $ such that $ C $ contains the vertices $ a,b,u,v $ and intersects $ S_{i} $ at $ \{u,v\} $. 
We modify the cycle $ C $ as follows (see Figure \ref{Fig:modification}). 
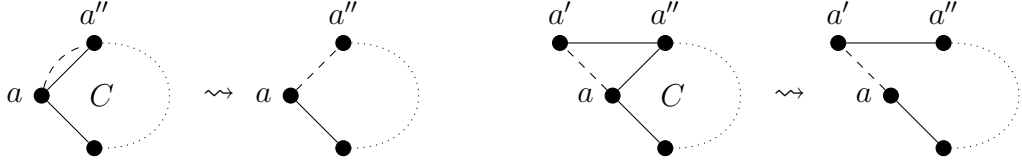
\begin{figure}[t]
\centering
\begin{tikzpicture}[baseline=-2]
\draw (0,0) node[v, label=left:{$ a $}](a){};
\draw (0.7,0.7) node[v, label=above:{$ a^{\prime\prime} $}](a2){};
\draw (0.7,-0.7) node[v](b){};
\draw (0.8,0) node(){$ C $};
\draw (b)--(a)--(a2);
\draw[bend left, dashed] (a) to (a2);
\draw[dotted] (a2) .. controls (2,0.7) and (2,-0.7) .. (b);
\end{tikzpicture}
$ \rightsquigarrow $
\begin{tikzpicture}[baseline=-2]
\draw (0,0) node[v, label=left:{$ a $}](a){};
\draw (0.7,0.7) node[v, label=above:{$ a^{\prime\prime} $}](a2){};
\draw (0.7,-0.7) node[v](b){};
\draw (b)--(a);
\draw[dashed] (a) to (a2);
\draw[dotted] (a2) .. controls (2,0.7) and (2,-0.7) .. (b);
\end{tikzpicture}
\hspace{10mm}
\begin{tikzpicture}[baseline=-2]
\draw (0,0) node[v, label=left:{$ a $}](a){};
\draw (0.7,0.7) node[v, label=above:{$ a^{\prime\prime} $}](a2){};
\draw (-0.7,0.7) node[v, label=above:{$ a^{\prime} $}](a1){};
\draw (0.7,-0.7) node[v](b){};
\draw (0.8,0) node(){$ C $};
\draw (b)--(a)--(a2)--(a1);
\draw[dashed] (a)--(a1);
\draw[dotted] (a2) .. controls (2,0.7) and (2,-0.7) .. (b);
\end{tikzpicture}
$ \rightsquigarrow $
\begin{tikzpicture}[baseline=-2]
\draw (0,0) node[v, label=left:{$ a $}](a){};
\draw (0.7,0.7) node[v, label=above:{$ a^{\prime\prime} $}](a2){};
\draw (-0.7,0.7) node[v, label=above:{$ a^{\prime} $}](a1){};
\draw (0.7,-0.7) node[v](b){};
\draw (b)--(a);
\draw (a2)--(a1);
\draw[dashed] (a)--(a1);
\draw[dotted] (a2) .. controls (2,0.7) and (2,-0.7) .. (b);
\end{tikzpicture}
\caption{The modification of a cycle}\label{Fig:modification}
\end{figure}
Let $ a^{\prime\prime} $ be a vertex adjacent to $ a $ in $ C $.
By Theorem \ref{Ibarra 1}, we have $ a^{\prime\prime} \in C_{1} $.  
If $ G $ has the negative edge $ \{a,a^{\prime\prime}\} $, then we replace the positive edge $ \{a,a^{\prime\prime}\} $ of $ C $ by the negative edge $ \{a,a^{\prime\prime}\} $. 
When $ G $ does not have the negative edge $ \{a,a^{\prime\prime}\} $, we replace the positive edge $ \{a,a^{\prime\prime}\} $ of $ C $ by the path consisting of the negative edge $ \{a,a^{\prime}\} $ and the positive edge $ \{a^{\prime},a^{\prime\prime}\} $. 
We make a similar modification with respect to the vertex $ b $. 
As a result, our cycle $ C $ has been modified so that it satisfies the assumptions in Lemma \ref{negative chord}. 
Therefore we have $ \{u,v\} \in E_{G}^{-} $ and hence $ G[S_{i}] $ is a complete signed graph with loops. 
\end{proof}

\begin{lemma}\label{dominating}
Suppose that $ N_{G^{-}[C_{1}]}(a) \subseteq S_{i} $. 
Assume that one of the following conditions holds:
\begin{enumerate}[(i)]
\item \label{dominating i} $ N_{G^{-}[C_{k}]}(b) \not\subseteq S_{i} $. 
\item \label{dominating ii} $ N_{G^{-}[C_{1}]}(a) \subseteq N_{G^{-}[C_{k}]}(b) $. 
\item \label{dominating iii} $ N_{G^{-}[C_{k}]}(b) \subseteq S_{i}, N_{G^{-}[C_{1}]}(a) \not\subseteq N_{G^{-}[C_{k}]}(b) $, and $ N_{G^{-}[C_{1}]}(a) \not\supseteq N_{G^{-}[C_{k}]}(b) $. 
\end{enumerate}
Then the following hold: 
\begin{enumerate}[(1)]
\item \label{dominating 1} Every vertex in $ N_{G^{-}[C_{1}]}(a) $ is dominating in $ G^{-}[S_{i}] $, that is, for any $ u \in N_{G^{-}[C_{1}]}(a) $ and another vertex $ v \in S_{i} $, there exists the negative edge $ \{u,v\} $. 
\item \label{dominating 2} Every vertex in $ N_{G^{-}[C_{1}]}(a) $ is dominating in $ G^{-}[S_{2}] $. 
\end{enumerate}
\end{lemma}
\begin{proof}
(\ref{dominating 1}) 
Take vertices $ u \in N_{G^{-}[C_{1}]}(a) $ and $ v \in S_{i} \setminus \{u\} $. 
We will show that there exists a negative edge $ \{u,v\} $. 

First, we assume (\ref{dominating i}). 
Then there exists a vertex $ b^{\prime} \in C_{k}\setminus S_{i} $ such that $ \{b,b^{\prime}\} \in E_{G}^{-} $. 
By Corollary \ref{Menger cor}, we obtain a cycle $ C $ of $ G^{+} $ containing $ a,b,u,v $ and intersecting $ S_{i} $ at $ \{u,v\} $. 
Replace the positive edge $ \{a,u\} $ of $ C $ by the negative edge $ \{a,u\} $ and make a modification with respect to $ b $ as Figure \ref{Fig:modification}. 
Then our modified cycle satisfies the conditions in Lemma \ref{negative chord} and hence we have $ \{u,v\} \in E_{G}^{-} $. 

Second, assume (\ref{dominating ii}). 
Note that $ u \in N_{G^{-}[C_{1}]}(a) $ implies $ \{a,u\}, \{u,b\} \in E_{G}^{-} $. 
By Corollary \ref{Menger cor} again, we obtain a cycle of $ G^{+} $ containing $ a,b,u,v $ and intersecting $ S_{i} $ at $ \{u,v\} $. 
Replace the positive edges $ \{a,u\}, \{u,b\} $ of $ C $ by the negative edges $ \{a,u\}, \{u,b\} $. 
The modified cycle satisfies the conditions in Lemma \ref{negative chord} and hence we have $ \{u,v\} \in E_{G}^{-} $. 

Finally, assume (\ref{dominating iii}). 
If $ u \in N_{G^{-}[C_{1}]}(a) \cap N_{G^{-}[C_{k}]}(b) $, then $ u $ is dominating in $ G^{-}[S_{i}] $ as in the case (\ref{dominating ii}). 
We assume that $ u \in N_{G^{-}[C_{1}]}(a) \setminus N_{G^{-}[C_{k}]}(b) $ and $ v \in N_{G^{-}[C_{k}]}(b) \setminus N_{G^{-}[C_{1}]}(a) $. 
Using Corollary \ref{Menger cor}, we have a cycle $ C $ of $ G^{+} $ containing $ a,b,u,v $ and intersecting $ S_{i} $ at $ \{u,v\} $. 
Replace the positive edges $ \{a,u\},\{v,b\} $ of $ C $ by the negative edges $ \{a,u\},\{v,b\} $. 
The modified cycle satisfies the conditions in Lemma \ref{negative chord} and hence we have $ \{u,v\} \in E_{G}^{-} $. 
Take a vertex $ w \in S_{i}\setminus \{u,v\} $. 
By Theorem \ref{Menger}, there exists an induced path $ p $ of $ G^{+} $ from $ a $ to $ b $ such that it intersects $ S $ at $ \{w\} $. 
Make a cycle $ C $ by connecting the path $ p $ and the path $ bvua $, where $ \{b,v\}, \{u,a\} $ are negative and $ \{v,u\} $ is positive (see Figure \ref{Fig:the balanced cycle}). 
\begin{figure}[t]
\centering
\begin{tikzpicture}
\draw (0,1.7) node[v, label=above:{$ w $}](w){};
\draw (-1.2,0.7) node[v, label=left:{$ a $}](a){};
\draw ( 1.2,0.7) node[v, label=right:{$ b $}](b){};
\draw (-0.5,0) node[v,label=below:{$ u $}](u){};
\draw ( 0.5,0) node[v,label=below:{$ v $}](v){};
\draw (0,0.8) node(){$ C $};
\draw[bend left, dotted] (a) to (w);
\draw[bend right, dotted] (b) to (w);
\draw[dashed] (a)--(u);
\draw[dashed] (b)--(v);
\draw (u)--(v);
\end{tikzpicture}
\caption{The balanced cycle (the dotted segments form the induced path $ p $)}\label{Fig:the balanced cycle}
\end{figure}
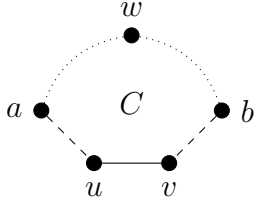
The balanced cycle $ C $ is of length at least five and hence must have a balanced chord. 
There are no edges connecting non-consecutive vertices in $ p $ since $ p $ is induced, and there are no negative edges $ \{u,b\}, \{v,a\} $ by the choice of $ u,v $. 
Hence candidates of negative chords are the negative edges connecting an internal vertex of $ p $ to $ u $ or $ v $. 
One can prove that such negative edges exist by considering smaller balanced cycles. 
Especially, we have $ \{u,w\} \in E_{G}^{-} $. 
Therefore we conclude that $ u $ is dominating in $ G^{-}[S_{i}] $. 

(\ref{dominating 2}) 
We proceed by induction on $ k $. 
When $ k=2 $, the thesis follows by (\ref{dominating 1}) since $ S_{1} $ is a unique minimal $ (a,b) $-separator. 
We assume that $ k \geq 3 $. 
If $ i=1 $, the assertion holds from (\ref{dominating 1}). 
Hence we assume that $ i \geq 2 $. 
By Theorem \ref{Ibarra2}(\ref{Ibarra2-2}), the separators form an antichain and hence there exists a vertex $ b^{\prime} \in S_{i} \setminus S_{i-1} $. 
Consider the subpath $ P^{\prime} $ of $ P $ from $ C_{1} $ to $ C_{i-1} $. 
We have that $ b^{\prime} \in S_{i} \subseteq C_{i-1}\setminus S_{i-1} $. 

Let $ S_{j} \ (1 \leq j \leq i-1) $ be a minimal $ (a,b^{\prime}) $-separator of minimal cardinality.
Using the clique intersection property (Theorem \ref{Ibarra2}(\ref{Ibarra2-1})), we have that $ N_{G^{-}[C_{1}]}(a) \subseteq C_{1} \cap S_{i} \subseteq S_{j} $. 

By (\ref{dominating 1}), every vertex $ u \in N_{G^{-}[C_{1}]}(a) $ is dominating in $ G^{-}[S_{i}] $. 
In particular, we have $ \{u,b^{\prime}\} \in E_{G}^{-} $. 
Hence $ N_{G^{-}[C_{1}]}(a) \subseteq N_{G^{-}[C_{i-1}]}(b^{\prime}) $ since $ N_{G^{-}[C_{1}]}(a) \subseteq S_{i} \subseteq C_{i-1} $. 
 
Therefore the subpath $ P^{\prime} $ and the vertices $ a,b^{\prime} $ satisfy the condition of the assertion. 
By our induction hypothesis, we conclude that the assertion is true. 
\end{proof}

\begin{lemma}\label{comp_signed or ss}
One of the following holds: 
\begin{enumerate}[(1)]
\item \label{comp_signed or ss 1} $ G $ has a signed-simplicial vertex. 
\item \label{comp_signed or ss 2} There exists a minimal vertex separator $ S $ of $ G^{+} $ such that $ G[S] $ is a complete signed graph with loops. 
\end{enumerate}
\end{lemma}
\begin{proof}
By Proposition \ref{sink_box_leaf}, every sink box of $ \mathcal{G} $ has at least two clique nodes which are leaves. 
We assume the clique nodes $ C_{1},C_{k} $ of our path $ P $ are leaves of a sink box. 
Moreover, suppose that the degrees $ \deg_{G^{-}[C_{1}]}(a), \deg_{G^{-}[C_{k}]}(b) $ are minimal in $ C_{1}\setminus S_{2}, C_{k}\setminus S_{k} $, respectively. 
If $ N_{G^{-}[C_{1}]}(a) \not\subseteq S_{i} $ and $ N_{G^{-}[C_{k}]}(b) \not\subseteq S_{i} $, then $ G[S_{i}] $ is a complete signed graph with loops by Lemma \ref{separator comp signed}. 
Thus (\ref{comp_signed or ss 2}) holds. 

Now, we may assume that $ N_{G^{-}[C_{1}]}(a) \subseteq S_{i} $ by symmetry. 
Then we have $ N_{G^{-}[C_{1}]}(a) \subseteq S_{2} $ by the clique intersection property (see Theorem \ref{Ibarra2}(\ref{Ibarra2-1})). 
We will show that $ \deg_{G^{-}[C_{1}]}(a) $ is minimal in $ C_{1} $. 
In order to do that, take a vertex $ u \in S_{2} $ and compare the degrees. 

First, assume that $ u \in N_{G^{-}[C_{1}]}(a) $. 
Since one of the conditions in Lemma \ref{dominating} holds,  the vertex $ u $ is dominating in $ G^{-}[S_{2}] $ by Lemma \ref{dominating}(\ref{dominating 2}). 
Then we have 
\begin{align*}
\deg_{G^{-}[C_{1}]}(u) \geq |\{a\} \cup (S_{2} \setminus \{u\})| = |S_{2}| \geq |N_{G^{-}[C_{1}]}(a)| = \deg_{G^{-}[C_{1}]}(a). 
\end{align*}

Second, suppose that $ u \in S_{2} \setminus N_{G^{-}[C_{1}]}(a) $. 
There exists a negative edge from each vertex in $ N_{G^{-}[C_{1}]}(a) $ to $ u $ since every vertex in $ N_{G^{-}[C_{1}]}(a) $ is dominating in $ G^{-}[S_{2}] $. 
Hence we have 
\begin{align*}
\deg_{G^{-}[C_{1}]}(u) \geq |N_{G^{-}[C_{1}]}(a)| = \deg_{G^{-}[C_{1}]}(a). 
\end{align*}

Thus $ \deg_{G^{-}[C_{1}]}(a) $ is minimal in $ C_{1} $. 
By Proposition \ref{bc threshold} and Corollary \ref{deg minimal ss}, we have that the vertex $ a $ is signed simplicial in $ G[C_{1}] $. 
Since $ a \in C_{1}\setminus S_{2} $ and $ C_{1} $ is a leaf of a sink box, the separator $ S_{2} $ separates $ a $ from any vertices in $ V_{G} \setminus C_{1} $ by Theorem \ref{Ibarra 1}. 
Therefore the vertices adjacent to $ a $ belong to $ C_{1} $ and hence $ a $ is a signed-simplicial vertex of $ G $. 
\end{proof}

\begin{example}
By Theorems \ref{Dirac minmal vertex separator} and \ref{Dirac simplicial}, every chordal simple graph has a simplicial vertex and all minimal vertex separators of it are cliques. 
Regarding our case, there are signed graphs which satisfy only one of the conditions in Lemma \ref{comp_signed or ss}. 
The left graph in Figure \ref{example bc comp ss} has signed-simplicial vertices $ b,e $ and a unique minimal vertex separator $ \{a,c,d\} $ does not induce a complete signed graph with loops. 
The right graph has no signed-simplicial vertices and the minimal vertex separator $ \{c,d\} $ induces $ B_{2} $. 
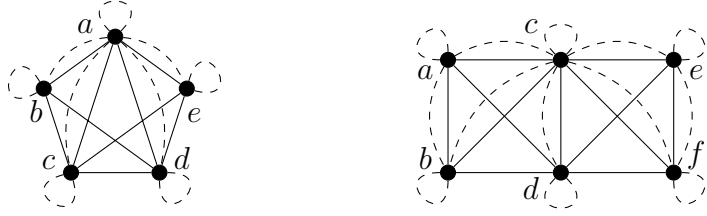
\begin{figure}[t]
\centering
\begin{tikzpicture}
\draw (90:1) node[v, label={[xshift=-4mm, yshift=-2mm]$ a $}](a){};
\draw (90+72:1) node[v, label={[xshift=-1mm, yshift=-7mm]$ b $}](b){};
\draw (90+72*2:1) node[v, label={[xshift=-3mm, yshift=-2mm]$ c $}](c){};
\draw (90+72*3:1) node[v, label={[xshift=3mm, yshift=-2mm]$ d $}](d){};
\draw (90+72*4:1) node[v, label={[xshift=1mm, yshift=-7mm]$ e $}](e){};
\draw (b)--(a)--(c)--(b)--(d)--(c)--(e)--(d)--(a)--(e);
\draw[bend right,dashed] (a) to (b);
\draw[bend right,dashed] (a) to (c);
\draw[bend left,dashed] (a) to (d);
\draw[bend left,dashed] (a) to (e);
\draw[scale=1.5,dashed] (a)  to[in=135,out=45,loop] (a);
\draw[scale=1.5,dashed] (b)  to[in=135+75,out=45+75,loop] (b);
\draw[scale=1.5,dashed] (c)  to[in=135+75*2,out=45+75*2,loop] (c);
\draw[scale=1.5,dashed] (d)  to[in=135+75*3,out=45+75*3,loop] (d);
\draw[scale=1.5,dashed] (e)  to[in=135+75*4,out=45+75*4,loop] (e);
\end{tikzpicture}
\hspace{15mm}
\begin{tikzpicture}
\draw (0,1.5) node[v,label={[xshift=-3mm, yshift=-5mm]$ a $}](a){}; 
\draw (0,0) node[v,label={[xshift=-3mm, yshift=-2mm]$ b $}](b){};
\draw (1.5,1.5) node[v,label={[xshift=-4mm, yshift=1mm]$ c $}](c){}; 
\draw (1.5,0) node[v,label={[xshift=-4mm, yshift=-6mm]$ d $}](d){};
\draw (3,1.5) node[v,label={[xshift=3mm, yshift=-5mm]$ e $}](e){}; 
\draw (3,0) node[v,label={[xshift=3mm, yshift=-2mm]$ f $}](f){};
\draw (a)--(b)--(d)--(f)--(e)--(c)--(a);
\draw (a)--(d)--(e);
\draw (b)--(c)--(f);
\draw (c)--(d);
\draw[bend right,dashed] (a) to (b);
\draw[bend left,dashed] (a) to (c);
\draw[bend right,dashed] (c) to (b);
\draw[bend right,dashed] (c) to (d);
\draw[bend left,dashed] (c) to (e);
\draw[bend left,dashed] (c) to (f);
\draw[bend left,dashed] (e) to (f);
\draw[scale=1.5,dashed] (a)  to[in=180,out=90,loop] (a);
\draw[scale=1.5,dashed] (b)  to[in=270,out=180,loop] (b);
\draw[scale=1.5,dashed] (c)  to[in=135,out=45,loop] (c);
\draw[scale=1.5,dashed] (d)  to[in=315,out=225,loop] (d);
\draw[scale=1.5,dashed] (e)  to[in=90,out=0,loop] (e);
\draw[scale=1.5,dashed] (f)  to[in=0,out=270,loop] (f);
\end{tikzpicture}
\caption{Examples for Lemma \ref{comp_signed or ss}}\label{example bc comp ss}
\end{figure}
\end{example}

The following Lemma is a generalization of Theorem \ref{Dirac simplicial}. 
\begin{lemma}\label{two non-adj div}
The signed graph $ G $ has at least two non-adjacent divisional vertices $ v_{1},v_{2} $ such that the contractions $ G/e_{i} $ with respect to some incident divisional edges $ e_{i}=(v_{i},v_{i}^{\prime}) $ satisfy the following conditions:
\begin{enumerate}[(i)]
\item $ (G/e_{i})^{+} \supseteq (G/e_{i})^{-} $. 
\item $ G/e_{i} $ is balanced chordal. 
\end{enumerate}
\end{lemma}
\begin{proof}
Without loss of generality, we may assume that $ G $ is connected. 
We proceed by induction on the number of vertices of $ G $. 
If $ |V_{G}| $ equals $ 1 $ or $ 2 $, then $ G^{+} $ must be complete, and hence we have nothing to prove. 

Assume that $ |V_{G}| \geq 3 $. 
By Lemma \ref{comp_signed or ss}, our graph $ G $ has a signed-simplicial vertex or a separator $ S $ such that $ G[S] $ is a complete signed graph with loops. 

First, suppose that $ G $ has a signed-simplicial vertex $ v_{1} $. 
Let $ F \coloneqq G\setminus \{v_{1}\} $, which is connected since $ G $ is connected and $ v_{1} $ is signed simplicial. 
The vertex $ v_{1} $ is divisional by Proposition \ref{ss => divisional} and every incident edge $ e_{1}=(v_{1},v_{1}^{\prime}) $ is divisional. 
Moreover, by Proposition \ref{add ss}(\ref{add ss 1}), the contraction $ G/e_{1} = F $ satisfies the conditions. 
We will show that there exists a vertex of $ F $ which is non-adjacent to $ v_{1} $ and divisional in $ F $. 
If $ F^{+} $ is complete, then every vertex in $ F $ is divisional in $ F $ by Propositions \ref{G+ comp and full loops} and \ref{ss => divisional} and the contraction of $ F $ with respect to an incident divisional edge satisfies the conditions. 
Since $ G^{+} $ is non-complete, there exists a vertex in $ F $ which is non-adjacent to $ v_{1} $. 
When $ F^{+} $ is non-complete, by the induction hypothesis, $ F $ has two non-adjacent divisional vertices such that the conditions are satisfied. 
Since $ v_{1} $ is signed simplicial, one of them is non-adjacent to $ v_{1} $. 
Thus, in the both cases, there exists a vertex $ v_{2} $ of $ F $ which is non-adjacent to $ v_{1} $ and there exists an incident divisional edge $ e_{2}=(v_{2},v_{2}^{\prime}) $ such that $ F/e_{2} $ satisfies the conditions. 

Now, we show that $ v_{2} $ is also divisional in $ G $ and $ G/e_{2} $ satisfies the conditions. 
By the definition of a divisional edge, there exists a non-negative integer $ d^{\prime} $ such that  
\begin{align*}
\chi(F,t) = (t-d^{\prime})\chi(F/e_{2},t). 
\end{align*}
Since $ v_{1},v_{2} $ are non-adjacent, we have that $ v_{1} $ is signed simplicial in $ G/e_{2} $. 
Let $ d $ denote the degree of $ v_{1} $ in $ G $ (or equivalently in $ G/e_{2} $). 
By Proposition \ref{add ss}(\ref{add ss 3}), we have 
\begin{align*}
\chi(G/e_{2},t)=(t-d)\chi((G/e_{2})\setminus \{v_{1}\},t)=(t-d)\chi(F/e_{2},t). 
\end{align*}
By Proposition \ref{add ss}(\ref{add ss 3}) again, we have that 
\begin{align*}
\chi(G,t) 
&= (t-d)\chi(G \setminus\{v_{1}\},t) 
= (t-d)\chi(F,t) \\
&= (t-d)(t-d^{\prime})\chi(F/e_{2},t) 
= (t-d^{\prime})\chi(G/e_{2},t). 
\end{align*}
Thus $ e_{2} $ is a divisional edge in $ G $.
Moreover, since $ v_{1} $ is signed simplicial in $ G/e_{2} $ and $ (G/e_{2}) \setminus \{v_{1}\} = F/e_{2} $ is balanced chordal, we have $ G/e_{2} $ is also balanced chordal by Proposition \ref{add ss}(\ref{add ss 4}). 
The condition $ (G/e_{2})^{+} \supseteq (G/e_{2})^{-} $ follows immediately from the condition $ (F/e_{2})^{+} \supseteq (F/e_{2})^{-} $.

Next, we suppose that $ G $ has a minimal vertex separator $ S $ such that $ G[S] $ is a complete signed graph with loops. 
Let $ A $ be the vertex set of a connected component of $ G \setminus S $. 
Put $ G_{1} \coloneqq G[A\cup S] $. 
We have $ G_{1}^{+} $ is complete or, by the induction hypothesis, $ G_{1} $ has two non-adjacent divisional vertices such that the condition satisfied. 
In the both cases, there exist a vertex $ v \in A $ and an incident divisional edge $ e=(v,v^{\prime}) $ of $ G_{1} $ such that $ G_{1}/e $ satisfies the conditions. 

We show that $ v $ is divisional in $ G $. 
By the definition of a divisional edge, there exists a non-negative integer $ d $ such that $ \chi(G_{1},t)=(t-d)\chi(G_{1}/e,t) $. 
Let $ G_{2} \coloneqq G \setminus (A\cup S) $. 
Since $ v \in A $, we have
\begin{align*}
G = G_{1} \cup G_{2} &\text{ and } G_{1} \cap G_{2} = G[S], \\
G/e = (G_{1}/e) \cup G_{2} &\text{ and } (G_{1}/e) \cap G_{2} = G[S]. 
\end{align*}
By Lemma \ref{gluing}, 
\begin{align*}
\chi(G,t) &= \frac{\chi(G_{1},t)\chi(G_{2},t)}{\chi(G[S],t)} 
= \frac{(t-d)\chi(G_{1}/e,t)\chi(G_{2},t)}{\chi(G[S],t)} 
= (t-d)\chi(G/e,t). 
\end{align*}
Thus $ e $ is divisional in $ G $, and hence $ v $ is divisional in $ G $. 
Moreover, since $ G_{1}/e $ and $ G_{2} $ are balanced chordal, we have $ G/e $ is also balanced chordal by Lemma \ref{gluing}. 
The condition $ (G/e)^{+} \supseteq (G/e)^{-} $ is obvious. 

We have shown that every connected component of $ G \setminus S $ has a vertex which is divisional in $ G $ having a divisional edge which satisfies the condition. 
Hence $ G $ has at least two non-adjacent divisional vertices satisfying the desired property. 
\end{proof}

\subsection{Proof of Theorem \ref{main theorem}}
We  are now ready to give a proof of Theorem \ref{main theorem}. 
\begin{proof}[Proof of Theorem \ref{main theorem}]
$ (\ref{main theorem 1}) \Rightarrow (\ref{main theorem 2}) $ 
If $ G^{+} $ is complete, then $ \mathcal{A}(G) $ is free by Proposition \ref{G+ comp and full loops}. 
Suppose that $ G^{+} $ is non-complete. 
We proceed by induction on the number of vertices $ |V_{G}| $. 
We may assume that $ |V_{G}| \geq 3 $. 
By Lemma \ref{two non-adj div}, we have a divisional edge $ e $ such that $ (G/e)^{+} \supseteq (G/e)^{-} $ and $ G/e $ is balanced chordal. 
Moreover, every vertex in $ G/e $ admits a loop. 
Therefore, by the induction hypothesis, we have that $ \mathcal{A}(G/e) $ is free. 
Using Lemma \ref{division theorem for signed graph}, we conclude that $ \mathcal{A}(G) $ is free. 

$ (\ref{main theorem 2}) \Rightarrow (\ref{main theorem 3}) $ 
This is trivial. 

$ (\ref{main theorem 3}) \Rightarrow (\ref{main theorem 1}) $ 
This follows by Lemma \ref{free => balanced chordal}. 
\end{proof}

\section{Open Problems}
It appears that determining freeness of arbitrary signed-graphic arrangements is quite difficult. 
In this paper, we consider the problem under the assumption $ G^{+} \supseteq G^{-} $. 
However, behavior of loop sets admitting freeness is still inscrutable. 
\begin{problem}
Assume that a balanced-chordal signed graph $ G=(G^{+},G^{-}) $ satisfies $ G^{+} \supseteq G^{-} $. 
Find a necessary and sufficient condition on a loop set $ L $ for $ \mathcal{A}(G^{+},G^{-},L) $ to be free. 
\end{problem}

For a free arrangement $ \mathcal{A} $, the multiset of degrees of a homogeneous basis for $ D(\mathcal{A}) $ is called the set of \textbf{exponents}. 
The degrees of a free graphic arrangement can be described in terms of graphs. 
Namely, if $ (v_{1}, \dots, v_{\ell}) $ is a perfect elimination ordering of a chordal graph $ G $, then the multiset $ \Set{\deg_{G[\{v_{1},\dots, v_{i} \}]}(v_{i}) | 1 \leq i \leq \ell} $ coincides with the exponents of $ \mathcal{A}(G) $. 

\begin{problem}
Is there any graphical interpretation of the exponents of free signed-graphic arrangements? 
\end{problem}

\bibliographystyle{amsplain1}
\bibliography{bibfile_paper}

\end{document}